\documentclass[11pt,leqno]{article}
\usepackage{graphicx, amsfonts, amsthm, amsxtra, amssymb, verbatim, makeidx}
\usepackage{subeqnarray, relsize}
\usepackage[mathscr]{euscript}
\usepackage{hyperref}
\makeindex
\makeglossary
\usepackage{amsmath,amssymb, amsthm}
\usepackage[utf8]{inputenc}
\usepackage[activeacute, english]{babel}
\usepackage{wrapfig}
\usepackage{amssymb, amsmath, amsthm} 
\usepackage{graphicx}
\usepackage{color}
\usepackage{amssymb}
\usepackage{url}
\usepackage{pdfpages}
\usepackage{fancyhdr}
\usepackage{hyperref}
\usepackage{subfig}
\usepackage{mathrsfs}
\usepackage{float}
\usepackage{pgf,tikz}
\usetikzlibrary{arrows}
\usetikzlibrary[patterns]

\DeclareMathOperator{\tr}{tr}
\DeclareMathOperator{\pn}{{\sf p}}

\newcommand{\CH}{\text{CH}}
\newcommand{\res}{\text{res}}
\newcommand{\Jac}{\text{Jac}}
\newcommand{\Hess}{\text{Hess}}
\newcommand{\N}{\mathbb{N}}
\newcommand{\Z}{\mathbb{Z}}
\newcommand{\C}{\mathbb{C}}
\newcommand{\Q}{\mathbb{Q}}
\newcommand{\R}{\mathbb{R}}
\newcommand{\K}{\mathbb{K}}
\renewcommand{\P}{\mathbb{P}}
\newcommand{\T}{\mathbb{T}}
\newcommand{\U}{\mathcal{U}}
\newcommand{\Hip}{\mathbb{H}}

\renewcommand{\Im}{\mathrm{I}\text{m}}

\newcommand{\dR}{\text{dR}}

\newcommand{\prim}{\text{prim}}

\newcommand{\p}[3]{\left #1 #3 \right #2}

\theoremstyle{theorem}

\newtheorem{thm}{Theorem}[section]

\newtheorem{lemma}[thm]{Lemma}

\newtheorem{cor}[thm]{Corollary}

\newtheorem{prop}[thm]{Proposition}

\theoremstyle{definition}

\newtheorem{rmk}{Remark}[section]

\newtheorem{notation}{Notation}[section]

\newtheorem{dfn}{Definition}[section]

\renewenvironment{proof}{{\bfseries \noindent Proof} }{ \qed \\}

\newcount\colveccount 
\newcommand*\colvec[1]{ 
    \global\colveccount#1
    \begin{pmatrix} 
    \colvecnext 
}

\def\colvecnext#1{
    #1
    \global\advance\colveccount-1 
    \ifnum\colveccount>0 
        \\
        \expandafter\colvecnext 
    \else 
        \end{pmatrix}
    \fi
}

\begin{document}


\def\gru{\mu} 
\def\pg{{ \sf S}}               
\def\TS{\mathlarger{\bf T}}                
\def\NB{{\mathlarger{\bf N}}}
\def\group{{\sf G}}
\def\NLL{{\rm NL}}   

\def\plc{{ Z_\infty}}    
\def\pola{{u}}      
\newcommand\licy[1]{{\mathbb P}^{#1}} 
\newcommand\aoc[1]{Z^{#1}}     
\def\HL{{\rm Ho}}     
\def\NLL{{\rm NL}}   

\def\Z{\mathbb{Z}}                   
\def\Q{\mathbb{Q}}                   
\def\C{\mathbb{C}}                   
\def\N{\mathbb{N}}                   
\def\uhp{{\mathbb H}}                
\def\A{\mathbb{A}}                   
\def\dR{{\rm dR}}                    
\def\F{{\cal F}}                     
\def\Sp{{\rm Sp}}                    
\def\Gm{\mathbb{G}_m}                 
\def\Ga{\mathbb{G}_a}                 
\def\Tr{{\rm Tr}}                      
\def\tr{{{\mathsf t}{\mathsf r}}}                 
\def\spec{{\rm Spec}}            
\def\ker{{\rm ker}}              
\def\GL{{\rm GL}}                
\def\ker{{\rm ker}}              
\def\coker{{\rm coker}}          
\def\im{{\rm Im}}               
\def\coim{{\rm Coim}}            
\def\p{{\sf  p}}
\def\U{{\cal U}}   

\def\weig{{\nu}}
\def\r{{ r}}                       
\def\k{{\sf k}}                     
\def\ring{{\sf R}}                   
\def\X{{\sf X}}                      
\def\Ua{{   L}}                      
\def\T{{\sf T}}                      
\def\asone{{\sf A}}                  

\def\Ts{{\sf S}}
\def\cmv{{\sf M}}                    
\def\BG{{\sf G}}                       
\def\podu{{\sf pd}}                   
\def\ped{{\sf U}}                    
\def\per{{\bf  P}}                   
\def\gm{{  A}}                    
\def\gma{{\sf  B}}                   
\def\ben{{\sf b}}                    

\def\Rav{{\mathfrak M }}                     
\def\Ram{{\mathfrak C}}                     
\def\Rap{{\mathfrak G}}                     

\def\mov{{\sf  m}}                    
\def\Yuk{{\sf C}}                     
\def\Ra{{\sf R}}                      
\def\hn{{ h}}                         
\def\cpe{{\sf C}}                     
\def\g{{\sf g}}                       
\def\t{{\sf t}}                       
\def\pedo{{\sf  \Pi}}                  

\def\Der{{\rm Der}}                   
\def\MMF{{\sf MF}}                    
\def\codim{{\rm codim}}                
\def\dim{{\rm    dim}}                
\def\Lie{{\rm Lie}}                   

\def\u{{\sf u}}                       

\def\imh{{  \Psi}}                 
\def\imc{{  \Phi }}                  
\def\stab{{\rm Stab }}               
\def\Vec{{\rm Vec}}                 

\def\Fg{{\sf F}}     
\def\hol{{\rm hol}}  
\def\non{{\rm non}}  
\def\alg{{\rm alg}}  
\def\tra{{\rm tra}}  

\def\bcov{{\rm \O_\T}}       

\def\leaves{{\cal L}}        

\def\cat{{\cal A}}              
\def\im{{\rm Im}}               

\def\pn{{\sf p}}              
\def\Pic{{\rm Pic}}           
\def\free{{\rm free}}         
\def \NS{{\rm NS}}    
\def\tor{{\rm tor}}
\def\codmod{{\xi}}    

\def\GM{{\rm GM}}

\def\perr{{\sf q}}        
\def\perdo{{\cal K}}   
\def\sfl{{\mathrm F}} 
\def\sp{{\mathbb S}}  

\newcommand\diff[1]{\frac{d #1}{dz}} 
\def\End{{\rm End}}              

\def\sing{{\rm Sing}}            
\def\cha{{\rm char}}             
\def\Gal{{\rm Gal}}              
\def\jacob{{\rm jacob}}          
\def\tjurina{{\rm tjurina}}      
\newcommand\Pn[1]{\mathbb{P}^{#1}}   
\def\P{\mathbb{P}}
\def\Ff{\mathbb{F}}                  

\def\O{{\cal O}}                     

\def\ring{{\mathsf R}}                         
\def\R{\mathbb{R}}                   

\newcommand\ep[1]{e^{\frac{2\pi i}{#1}}}
\newcommand\HH[2]{H^{#2}(#1)}        
\def\Mat{{\rm Mat}}              
\newcommand{\mat}[4]{
     \begin{pmatrix}
            #1 & #2 \\
            #3 & #4
       \end{pmatrix}
    }                                
\newcommand{\matt}[2]{
     \begin{pmatrix}                 
            #1   \\
            #2
       \end{pmatrix}
    }
\def\cl{{\rm cl}}                

\def\hc{{\mathsf H}}                 
\def\Hb{{\cal H}}                    
\def\pese{{\sf P}}                  

\def\PP{\tilde{\cal P}}              
\def\K{{\mathbb K}}                  

\def\M{{\cal M}}
\def\RR{{\cal R}}
\newcommand\Hi[1]{\mathbb{P}^{#1}_\infty}
\def\pt{\mathbb{C}[t]}               
\def\gr{{\rm Gr}}                
\def\Im{{\rm Im}}                
\def\Re{{\rm Re}}                
\def\depth{{\rm depth}}
\newcommand\SL[2]{{\rm SL}(#1, #2)}    
\newcommand\PSL[2]{{\rm PSL}(#1, #2)}  
\def\Resi{{\rm Resi}}              

\def\L{{\cal L}}                     
\def\Aut{{\rm Aut}}              
\def\any{R}                          
\newcommand\ovl[1]{\overline{#1}}    

\newcommand\mf[2]{{M}^{#1}_{#2}}     
\newcommand\mfn[2]{{\tilde M}^{#1}_{#2}}     

\newcommand\bn[2]{\binom{#1}{#2}}    
\def\ja{{\rm j}}                 
\def\Sc{\mathsf{S}}                  
\newcommand\es[1]{g_{#1}}            
\newcommand\V{{\mathsf V}}           
\newcommand\WW{{\mathsf W}}          
\newcommand\Ss{{\cal O}}             
\def\rank{{\rm rank}}                
\def\Dif{{\cal D}}                   
\def\gcd{{\rm gcd}}                  
\def\zedi{{\rm ZD}}                  
\def\BM{{\mathsf H}}                 
\def\plf{{\sf pl}}                             
\def\sgn{{\rm sgn}}                      
\def\diag{{\rm diag}}                   
\def\hodge{{\rm Hodge}}
\def\HF{{ F}}                                
\def\WF{{ W}}                               
\def\HV{{\sf HV}}                                
\def\pol{{\rm pole}}                               
\def\bafi{{\sf r}}
\def\id{{\rm id}}                               
\def\gms{{\sf M}}                           
\def\Iso{{\rm Iso}}                           

\def\hl{{\rm L}}    
\def\imF{{\rm F}}
\def\imG{{\rm G}}

\begin{center}
{\LARGE\bf Periods of complete intersection algebraic cycles}

\vspace{.25in} {\large {\sc Roberto  Villaflor Loyola}}\footnote{
Instituto de Matem\'atica Pura e Aplicada, IMPA, Estrada Dona Castorina, 110, 22460-320, Rio de Janeiro, RJ, Brazil,
{\tt rvilla@impa.br}}
\end{center}


\def\pn{{\sf p}}
\def\rootsG{{\sf G}}
\def\NLL{{\rm NL}}   

\begin{abstract}
For every even number $n$, and every $n$-dimensional smooth hypersurface of $\P^{n+1}$ of degree $d$, we compute the periods of all its $\frac{n}{2}$-dimensional complete intersection algebraic cycles. Furthermore, we determine the image of the given algebraic cycle under the cycle class map inside the De Rham cohomology group of the corresponding hypersurface in terms of its Griffiths basis and the polarization. As an application, we use this information to address variational Hodge conjecture for a non complete intersection algebraic cycle. We prove that the locus of general hypersurfaces containing two
linear cycles whose intersection is of dimension less than $\frac{n}{2}-\frac{d}{d-2}$, corresponds to the Hodge
locus of any integral combination of such linear cycles. 
\end{abstract}

\section{Introduction}
Consider $X$ any smooth degree $d$ hypersurface of $\P^{n+1}$, and let us denote by $\theta\in H^{1,1}(X)\cap H^2(X,\Z)$ its polarization. From Lefschetz hyperplane section theorem it follows that the image of the cycle class map for codimension $k$ algebraic cycles is generated over $\Q$ by $\theta^k$ for $k\neq \frac{n}{2}$. In the case $n$ is even, it remains to determine the image of the cycle class map for $\frac{n}{2}$-dimensional algebraic cycles in $X$. The Hodge conjecture claims that this image corresponds to the space of Hodge cycles $H^{\frac{n}{2},\frac{n}{2}}(X)\cap H^n(X,\Q)$. Since the cycle class map of an algebraic cycle captures the cohomological information of the cycle, to describe its image is equivalent to determine all its periods. In order to compute the periods of an algebraic cycle, we restrict ourselves to the subgroup of $\CH^\frac{n}{2}(X)$ generated by the algebraic subvarieties $Z\subseteq X$ that are complete intersections inside $\P^{n+1}$. Our main result is the following:

\begin{thm}
\label{cycleclass}
Let $X\subseteq \P^{n+1}$ be a smooth degree $d$ hypersurface of even dimension $n$ given by $X=\{F=0\}$. Suppose that $Z:=\{f_1=\cdots=f_{\frac{n}{2}+1}=0\}\subseteq X$ is a complete intersection inside $\P^{n+1}$ and 
$$
I(Z)=\langle f_1,\ldots,f_{\frac{n}{2}+1}\rangle\subseteq \C[x_0,\ldots,x_{n+1}].
$$
Write
$$
F=f_1g_1+\cdots+f_{\frac{n}{2}+1}g_{\frac{n}{2}+1},
$$
and define 
$$
H=(h_0,\ldots,h_{n+1}):=(f_1,g_1,\ldots,f_{\frac{n}{2}+1},g_{\frac{n}{2}+1}).
$$
Then
$$
[Z]=\frac{\deg(Z)}{\deg(X)} \theta^\frac{n}{2}-\frac{\frac{n}{2}!}{\deg(X)}\res\left(\frac{\det(\Jac(H))\Omega}{F^{\frac{n}{2}+1}}\right)^{\frac{n}{2},\frac{n}{2}}\in H^{\frac{n}{2},\frac{n}{2}}(X),
$$
where $\Omega=\sum_{i=0}^{n+1}(-1)^ix_i dx_0\wedge\cdots\widehat{dx_i}\cdots\wedge dx_{n+1}$ is the generator of $H^0(\P^{n+1},\Omega_{\P^{n+1}}^{n+1}(n+2))$.
\end{thm}

The previous result follows from a direct computation of all the periods of the algebraic cycle $Z$ over a set of generators of the De Rham cohomology group of $X$, therefore is a consequence of the following result:

\begin{thm}
\label{15.5.2018.2}
Under the hypothesis of Theorem \ref{cycleclass}, let $J^F:=\langle \frac{\partial F}{\partial x_0},\ldots,\frac{\partial F}{\partial x_{n+1}}\rangle\subseteq \C[x_0,\ldots,x_{n+1}]$ be the Jacobian ideal associated to $F$. Then, for every homogeneous polynomial $P\in \C[x_0,\ldots,x_{n+1}]$ of degree $\deg(P)={(d-2)(\frac{n}{2}+1)}$
\begin{equation}
\label{19.10.2018}
\displaystyle\int_Z\res\left(\frac{P\Omega}{F^{\frac{n}{2}+1}}\right)=\frac{(2\pi\sqrt{-1})^{\frac{n}{2}}}{\frac{n}{2}!}c\cdot (d-1)^{n+2},    
\end{equation}
where $c\in \C$ is the unique number such that 
$$
P\cdot \det(\Jac(H))\equiv c\cdot \det(\Hess(F))\text{ (mod }J^F).
$$
\end{thm}

It is implicit in the statement of Theorem \ref{15.5.2018.2} that $P\cdot\det(\Jac(H))\in J^F+\langle \det(\Hess(F))\rangle$. In fact if we denote the \textit{Jacobian ring} by
$
R^F:={\C[x_0,\ldots,x_{n+1}]}/{J^F},
$
it satisfies $R^F_{(d-2)(n+2)}=\C\cdot \det(\Hess(F))$. This is consequence of a classical theorem due to Macaulay (see \S\ref{rochebday}, Theorem \ref{08.10.6}), which implies that $R^F$ is an Artinian Gorenstein algebra of socle $(d-2)(n+2)$. We will briefly discuss Artinian Gorenstein algebras and ideals in \S \ref{rochebday}. 

The advantage of working with periods instead of considering directly the cohomology classes is that the period equation \eqref{19.10.2018} depends continuously on the parameters $H=(h_0,\ldots,h_{n+1})$ (see \S\ref{2}, Proposition \ref{cont}). And so, we can perturb the pair $(Z,X)$ in order to reduce ourselves to the case where $Z$ is also smooth. After this reduction, the main idea in the proof of Theorem \ref{15.5.2018.2} is to construct a chain of smooth projective varieties $Z=Z_0\subseteq Z_1\subseteq\cdots\subseteq Z_{\frac{n}{2}+1}=\P^{n+1}$ where each $Z_i$ is a hypersurface of $Z_{i+1}$ given by the intersection of $Z_{i+1}$ with a very ample divisor of $\P^{n+1}$. Then using an explicit description of the coboundary map in \v{C}ech cohomology associated to the Poincar\'e residue sequence, we relate the periods of $Z_i$ with the periods of $Z_{i+1}$ (see \S\ref{1}, Proposition \ref{27.10.1}). And so, the period computation is reduced to a computation of an integral of a top form over $\P^{n+1}$, which is computed in \S\ref{1}, Corollary \ref{15.5.2018}. In \S \ref{3} we produce some applications of Theorem \ref{cycleclass} and Theorem \ref{15.5.2018.2}, getting computable formulas for the intersection of two $\frac{n}{2}$-dimensional complete intersection algebraic cycles inside $X$, in terms of their defining equations (see \S \ref{3}, Corollary \ref{4.8.2018.3}).

In Deligne's work on absolute Hodge cycles \cite{Deligne1982}, he showed that the periods of algebraic cycles belong to the field of definition of the variety and the corresponding algebraic cycle. Providing necessary conditions on the periods of Hodge cycles in order to satisfy the Hodge conjecture. Periods of algebraic cycles played a central role in the study of components of the Noether-Lefschetz locus by means of the infinitesimal variations of Hodge strucutres, leaded by Voisin \cite{voisin1988, voisin89, voisin90, voisin1991}, Green \cite{green1988, green1989}, Harris \cite{gri83III, CHM88} and many others \cite{lopez1991, kim1991, Otwinowska2003, Maclean05, kloosterman2007, Dan2017}. In 2014, Movasati reconsidered the problem of computing explicitly the periods of algebraic cycles. In \cite{GMCD-NL}, Movasati exposed several possible applications of these computations, among them a computational approach to certain special cases of variational Hodge conjecture. These ideas gave place to the computation of formulas for periods of linear cycles inside Fermat varieties appearing in \cite[Theorem 1]{MV} (these formulas can be deduced from Theorem \ref{15.5.2018.2}, see \S\ref{3}, Corollary \ref{perlinfer}). On the other hand, a parallel approach was considered by Sert\"oz in \cite{emre}, where he implemented an algorithm for approximating periods of arbitrary Hodge cycles inside hypersurfaces.

Following \cite{GMCD-NL}, we used the period formulas in \cite{MV} to handle variational Hodge conjecture for a non-complete intersection algebraic cycle inside the Fermat variety. Variational Hodge conjecture is a major conjecture proposed by Grothendieck in 1966, as a weak version of Hodge conjecture (see \cite[page 103]{gro66}). While Hodge conjecture claims that every Hodge cycle inside a smooth projective variety is an algebraic cycle. Variational Hodge conjecture claims that in all proper families of smooth projective varieties with connected base, a flat section of its de Rham cohomology bundle is an algebraic cycle at one point if and only if it is an algebraic cycle everywhere. In 1972, Bloch proved variational Hodge conjecture for deformations of algebraic cycles supported in local complete intersections which are semi-regular inside the corresponding smooth projective variety (see \cite{Bloch1972}). Semi-regularity is a strong condition, difficult to check in concrete examples (see \cite{DK2016} for a discussion about examples of semi-regular varieties). In 2003, Otwinowska considered variational Hodge conjecture for algebraic cycles inside smooth degree $d$ hypersurfaces $X$ of the projective space $\P^{n+1}$ of even dimension $n$. In this context, she proved (among several other remarkable results) that variational Hodge conjecture is satisfied for algebraic cycles supported in one $\frac{n}{2}$-dimensional complete intersection $Z$ of $\P^{n+1}$ contained in $X$, for $\deg(X)\gg 0$ (see \cite{Otwinowska2003}). An improvement of this result was presented by Dan, removing the condition on the degree (see \cite{Dan2017}). Despite Otwinowska and Dan's result, it is not known if the complete intersection subvarieties are semi-regular inside the corresponding hypersurface. The first explicit non-complete intersection algebraic cycle considered in high dimension hypersurfaces was treated in \cite[Theorem 2]{MV} with computer assistance. Several $\Z$-combinations of two linear cycles inside Fermat varieties where considered, but only some of them were proved to satisfy variational Hodge conjecture (by means of a first order approximation of the Hodge locus).

In the same spirit of \cite{MV}, we use Theorem \ref{15.5.2018.2} in order to analyze variational Hodge conjecture for cycles obtained as $\Z$-combinations of two linear cycles inside a hypersurface. We separate our analysis depending on the dimension of the intersection of the considered linear cycles. We generalize \cite[Theorem 2]{MV} (providing a theoretic proof) to arbitrary degree and dimension in the following way: 

\begin{thm}
\label{11.10.2018}
Let $X\subseteq\P^{n+1}$ be the Fermat variety of even dimension $n$ and degree $d$. Let $\P^\frac{n}{2}, \check\P^\frac{n}{2}\subseteq X$ be the two linear subvarieties such that $\P^\frac{n}{2}\cap \check\P^\frac{n}{2}=\P^m$ given by
$$
\P^{n-m}:=\{x_{n-2m}-\zeta_{2d}x_{n-2m+1}=\cdots=x_n-\zeta_{2d}x_{n+1}=0\},
$$
$$
\P^\frac{n}{2}:=\{x_0-\zeta_{2d}x_1=\cdots=x_{n-2m-2}-\zeta_{2d}x_{n-2m-1}=0\}\cap \P^{n-m},
$$
$$
\check\P^\frac{n}{2}:=\{x_0-\zeta_{2d}^{\alpha_0}x_1=\cdots=x_{n-2m-2}-\zeta_{2d}^{\alpha_{n-2m-2}}x_{n-2m-1}=0\}\cap \P^{n-m},
$$ 
where $\zeta_{2d}\in\C$ is a primitive $2d$-root of unity, and $\alpha_0,\alpha_2,\ldots,\alpha_{n-2m-2}\in \{3,5,\ldots,2d-1\}$. Then, for $m<\frac{n}{2}-\frac{d}{d-2}$, $a,b\in\Z\setminus\{0\}$ and $\delta:=a\cdot \P^\frac{n}{2}+b\cdot \check\P^{\frac{n}{2}}\in \CH^\frac{n}{2}(X)$ we have 
$$
V_{[\delta]}=V_{[\P^\frac{n}{2}]}\cap V_{[\check\P^\frac{n}{2}]},
$$
and the Hodge locus $V_{[\delta]}$ is smooth and reduced (see \S \ref{hodgelocus}, Definition \ref{5.11.2018}, for the definition of the Hodge locus). In particular, variational Hodge conjecture holds for $[\delta]\in H^n(X,\Z)\cap H^{\frac{n}{2},\frac{n}{2}}(X)$ in these cases. On the other hand, for $m\ge \frac{n}{2}-\frac{d}{d-2}$, the Zariski tangent space of $V_{[\delta]}$ has dimension strictly bigger than the dimension of $V_{[\P^\frac{n}{2}]}\cap V_{[\check\P^\frac{n}{2}]}$ (which is smooth and reduced, see \S \ref{4}, Proposition \ref{16.10.2018.2}).
\end{thm}

We remark that \cite[Theorem 2]{Maclean05} covers the case $(n,d)=(2,5)$. The main ingredient missing from \cite{MV} that allows us to prove Theorem \ref{11.10.2018} is the explicit computation of the cycle class map given in Theorem \ref{cycleclass}.
After the algebraicity of the locus of Hodge cycles proved by Cattani, Deligne and Kaplan \cite{cadeka}, we can state variational Hodge conjecture in the following local analytic format: \textit{``If $\delta_0\in H^n(X_0,\Z)\cap H^{\frac{n}{2},\frac{n}{2}}(X_0)$ is the cohomological class of an algebraic cycle, then $\delta_t\in H^n(X_t,\Z)\cap H^{\frac{n}{2},\frac{n}{2}}(X_t)$ is the cohomological class of an algebraic cycle for every $t\in V_{\delta_0}$.''} This version of variational Hodge conjecture is the one we are always referring to, in particular in Theorem \ref{11.10.2018}. Finally, by a simple argument informed by Movasati, we can deduce from Theorem \ref{11.10.2018} the following result confirming variational Hodge conjecture for combinations of linear cycles inside general hypersurfaces containing such cycles.

\begin{thm}
\label{gent}
Let $\pi:X\rightarrow T$ be the family of smooth degree $d$ hypersurfaces of $\P^{n+1}$, of even dimension $n$. Consider
$$
W_m:=\{t\in T: X_t\text{ contains two linear cycles }\P^\frac{n}{2}, \check\P^\frac{n}{2}\text{ with }\P^\frac{n}{2}\cap\check\P^\frac{n}{2}=\P^m\}.    
$$
If $m<\frac{n}{2}-\frac{d}{d-2}$, then for all $a,b\in\Z\setminus\{0\}$ variational Hodge conjecture holds for the Hodge cycle $[\delta]=a[\P^\frac{n}{2}]+b[\check\P^\frac{n}{2}]\in H^n(X_t,\Z)\cap H^{\frac{n}{2},\frac{n}{2}}(X_t)$, for $t\in W_m$ general.
\end{thm}

\section{Artinian Gorenstein algebras}
\label{rochebday}
As part of the algebraic background we need, we will state in this section some results about Artinian Gorenstein algebras. We begin with a classical result due to Macaulay (for a proof see \cite[Theorem 6.19]{vo03}).

\begin{thm}[Macaulay \cite{mac16}]
\label{08.10.6}
Given $f_0,\ldots,f_{n+1}\in\C[x_0,\ldots,x_{n+1}]$ homogeneous polynomials with $\deg(f_i)=d_i$ and 
$$
\{f_0=\cdots=f_{n+1}=0\}=\varnothing\subseteq \P^{n+1}.
$$
Let
$$
R:=\frac{\C[x_0,\ldots,x_{n+1}]}{\langle f_0,\ldots,f_{n+1}\rangle}.
$$
Then for $\sigma:=\sum_{i=0}^{n+1}(d_i-1)$, we have that 
\begin{itemize}
\item[(i)] $\text{dim}_\C\text{ } R_\sigma=1$.
\item[(ii)] For every $0\le i\le \sigma$ the multiplication map $$R_i\times R_{\sigma-i}\rightarrow R_\sigma$$ is a perfect pairing.
\item[(iii)] $R_e=0$ for $e>\sigma$.
\end{itemize}
\end{thm}

\begin{dfn}
Let $n\in \N$, and $I\subseteq \C[x_0,\ldots,x_{n+1}]$ an ideal. We say that the quotient ring $R:=\C[x_0,\ldots,x_{n+1}]/I$ is an \textit{Artinian Gorenstein algebra} if it satisfies items (i), (ii), (iii) of Macaulay Theorem \ref{08.10.6} for some $\sigma\in \N$. We say $\sigma$ is the \textit{socle of $R$} and denote it $\sigma=soc(R)$.
\end{dfn}

\begin{notation}
Despite the Artinian Gorenstein property is reserved for algebras, we will also say that \textit{$I$ is Artinian Gorenstein of socle $\sigma$}, when $R=\C[x_0,\ldots,x_{n+1}]/I$ is.
\end{notation}

\begin{rmk}
An elementary observation is that if $I$ is Artinian Gorenstein of socle $\sigma$, and $P\in \C[x_0,\ldots,x_{n+1}]_\mu\setminus I_\mu$, then the quotient ideal
$$
(I:P):=\{Q\in \C[x_0,\ldots,x_{n+1}]: PQ\in I\},
$$ 
is Artinian Gorenstein of socle $\sigma-\mu$. Is also elementary that if $I_1\subseteq I_2$ are two Artinian Gorenstein ideals of the same socle, then $I_1=I_2$.
\end{rmk}

We end this section with a proposition we will use in the proof of Theorem \ref{11.10.2018}.

\begin{prop}
\label{caex}
Consider the ideal $I:=\langle x_0^{d-1},\ldots,x_{2r-1}^{d-1}\rangle\subseteq \C[x_0,\ldots,x_{2r-1}]$. Let $d\ge 3$, and $\beta_1,\beta_2, c_1,c_2\in \C^\times$ with $\beta_1\neq \beta_2$. For $i=1,2$, define 
$$
R_i:=c_i\cdot\prod_{j=1}^r\frac{(x_{2j-2}^{d-1}-(\beta_ix_{2j-1})^{d-1})}{(x_{2j-2}-\beta_ix_{2j-1})}. 
$$
Then 
\begin{equation}
\label{3.12.2018}
(I:R_1)_e\cap (I:R_2)_e=(I:R_1+R_2)_e,    
\end{equation}
if and only if $e\neq (d-2)\cdot r$.
\end{prop}

\begin{proof}
First of all, note that $(I:R_1)$, $(I:R_2)$ and $(I:R_1+R_2)$ are Artinian Gorenstein ideals of socle $(d-2)\cdot r$. In consequence,
$$
(I:R_1)\cap (I:R_2)\neq (I:R_1+R_2).
$$
Otherwise, we would have $(I:R_1+R_2)\subseteq (I:R_1)$, which implies $$(I:R_1)=(I:R_1+R_2)=(I:R_2),$$ a contradiction. Therefore, in order to prove the proposition, it is enough to prove \eqref{3.12.2018} for $e\neq (d-2)\cdot r$. If $e>(d-2)\cdot r$, the equality \eqref{3.12.2018} is trivial since $(d-2)\cdot r$ is the socle of the three ideals. If $e<(d-2)\cdot r$, we claim \eqref{3.12.2018} reduces to the case $e=(d-2)\cdot r-1$. In fact, if we assume \eqref{3.12.2018} fails for some $e< (d-2)\cdot r$, we can choose 
\begin{equation}
\label{obvio1}
p\in (I:R_1+R_2)_{e}\setminus (I:R_1)_{e}.
\end{equation}
Since $(I:R_1)$ is Artinian Gorenstein of socle $(d-2)\cdot r$, the perfect pairing property implies that we can find a degree $(d-2)\cdot r-e$ monomial $x^i=x_0^{i_0}\cdots x_{2r-1}^{i_{2r-1}}$ such that 
\begin{equation}
\label{obvio2}
x^i\cdot p\in (I:R_1+R_2)_{(d-2)\cdot r}\setminus (I:R_1)_{(d-2)\cdot r}.
\end{equation}
Since $\deg (x^i)>0$, there exist some $i_j>0$, then \eqref{obvio1} and \eqref{obvio2} imply that
$$
\frac{x^i}{x_j}\cdot p\in (I:R_1+R_2)_{(d-2)\cdot r-1}\setminus (I:R_1)_{(d-2)\cdot r-1},
$$
and so \eqref{3.12.2018} would fail for $e=(d-2)\cdot r-1$, as claimed. Therefore, we just consider the case $e=(d-2)\cdot r-1$. It is enough to show that $(I:R_1+R_2)_e\subseteq (I:R_1)_e\cap (I:R_2)_e$.
Take $p\in (I:R_1+R_2)_e$. Without loss of generality we may assume it can be written as
$$
p=\sum_{k\text{ even}}\sum_{l=0}^{d-3}x_k^lx_{k+1}^{d-3-l}p_{k,l},
$$
where each $p_{k,l}$ does not depend on $x_k$ and $x_{k+1}$, and is a $\C$-linear combination of monomials of the form $x_0^{i_0}\cdots x_{k-1}^{i_{k-1}}x_{k+2}^{i_{k+2}}\cdots x_{2r-1}^{i_{2r-1}}$ with $i_{2j-2}+i_{2j-1}=d-2$, for all $j\in\{1,\ldots, r\}\setminus\{\frac{k}{2}+1\}$.
For every $k$ and $l$, and $i=1,2$, there exist a constant $a_{k,l, i}\in \C$ such that 
$$
p_{k,l}\frac{R_i}{(x_k^{d-2}+x_k^{d-3}(\beta_ix_{k+1})+\cdots+(\beta_ix_{k+1})^{d-2})}\equiv a_{k,l,i}\frac{(x_0\cdots x_{2r-1})^{d-2}}{(x_kx_{k+1})^{d-2}},
$$
modulo $\langle x_0^{d-1},\ldots,x_{k-1}^{d-1},x_{k+2}^{d-1},\ldots,x_{2r-1}^{d-1}\rangle$.
Then 
$$
pR_i\equiv (x_0\cdots x_{2r-1})^{d-2}\sum_{k\text{ even}}\left(\frac{1}{x_k}\sum_{l=0}^{d-3}a_{k,l,i}\beta_i^{l+1}+\frac{1}{x_{k+1}}\sum_{l=0}^{d-3}a_{k,l,i}\beta_i^l\right),
$$
modulo $I$. Since $p\cdot (R_1+R_2)\in I$ we conclude that
$$
\sum_{l=0}^{d-3}a_{k,l,1}\beta_1^{l+1}+\sum_{l=0}^{d-3}a_{k,l,2}\beta_2^{l+1}=\sum_{l=0}^{d-3}a_{k,l,1}\beta_1^l+\sum_{l=0}^{d-3}a_{k,l,2}\beta_2^l=0.
$$
Since $\beta_1\neq \beta_2$, this implies
$$
\sum_{l=0}^{d-3}a_{k,l,1}\beta_1^l=\sum_{l=0}^{d-3}a_{k,l,2}\beta_2^l=0,
$$
and so $pR_i\in I$ for $i=1,2$.
\end{proof}

\section{Cycle class map and periods}
Let us explain what we mean by periods of algebraic cycles inside smooth hypersurfaces.
Let $X$ be any smooth projective variety of dimension $n$, and $Z\in \CH^k(X)$ a codimension $k$ algebraic cycle of $X$. The \textit{cycle class map} can be factored as
$$
[\cdot]: \CH^k(X)\xrightarrow{\eta} H^{2n-2k}_\dR(X)^*\simeq H^{2k}_\dR(X)
$$
where the second map is given by the perfect pairing in De Rham cohomology induced by the integration over $X$ of the wedge product (divided by $(2\pi\sqrt{-1})^n$), and the former map corresponds to
$$
\eta_Z(\omega):=\frac{1}{(2\pi\sqrt{-1})^{n-k}}\int_Z\omega\in\C, \hspace{5mm}\forall \omega\in H^{2n-2k}_\dR(X).
$$
Note that for $k=1$ the cycle class map corresponds to the first Chern class. 

\begin{dfn}
Given an algebraic cycle $Z\in \CH^k(X)$, we say that the complex numbers $\eta_Z(\omega)$ are the \textit{periods of $Z$}, for all $\omega\in H^{2n-2k}_\dR(X)$.
\end{dfn}

\begin{rmk}
In general, for $X$ any smooth projective variety we have natural maps $H^{2k}(X,\Z)\rightarrow H^{2k}(X,\Q)\hookrightarrow H^{2k}(X,\C)\simeq H^{2k}_\dR(X)$. In spite the first map $H^{2k}(X,\Z)\rightarrow H^{2k}(X,\Q)$ is not in general injective, we will always denote by $H^{2k}(X,\Z)$ the cohomology with $\Z$-coefficients modulo torsion. Thus we will identify them without further mention as a chain of abelian groups
$$
H^{2k}(X,\Z)\subseteq H^{2k}(X,\Q)\subseteq H^{2k}(X,\C)=H^{2k}_\dR(X).
$$
Under this identification we will say that some $\omega\in H^{2k}_\dR(X)$ is an \textit{integral} (respectively \textit{rational}) \textit{class}, denoted $\omega\in H^{2k}_\dR(X)\cap H^{2k}(X,\Z)$ (respectively $\omega\in H^{2k}_\dR(X)\cap H^{2k}(X,\Q)$), if it only has integral (respectively rational) periods over $H_{2k}(X,\Z)$, i.e.
$$
\frac{1}{(2\pi\sqrt{-1})^k}\int_\delta \omega\in \Z \ , \ \ \ \forall \delta\in H_{2k}(X,\Z).
$$
\end{rmk}

Recalling Griffiths' work \cite{gr69}, in the case $X=\{F=0\}\subseteq\P^{n+1}$ is a smooth hypersurface of even dimension $n$ given by a homogeneous polynomial with $\deg F=d$, each piece of the Hodge filtration of $H^n_\dR(X)_{\prim}$ is generated by the differential forms 
$$
\omega_P:=\res\left(\frac{P\Omega}{F^{q+1}}\right)\in F^{n-q}H^n_\dR(X)_{\prim},
$$
for $P\in \C[x_0,\ldots,x_{n+1}]_{d(q+1)-n-2}$, where  $\Omega:=\iota_{\sum_{i=0}^{n+1}x_i\frac{\partial}{\partial x_i}}(dx_0\wedge\cdots\wedge dx_{n+1})=\sum_{i=0}^{n+1}(-1)^i x_i\widehat{dx_i},$
and $\res: H^{n+1}_\dR(\P^{n+1}\setminus X)\rightarrow H^n_\dR(X)$ is the \textit{residue map}.

\begin{notation}
\label{22.10.2018.2}
Whenever we are considering a set of 1-forms $\{y_i: i=1,\ldots,k\}$ we will use the notation $$\widehat{y_i}:=y_1\wedge \cdots\widehat{y_i}\cdots\wedge y_k.$$ This notation will be highly used in \S \ref{2}.
\end{notation}
We are interested in computing the periods of all $\frac{n}{2}$-dimensional algebraic cycles $Z\subseteq X$. Notice that, since $Z$ is a projective variety of positive dimension, it intersects every divisor of $X$, so it is impossible to find an affine chart of $X$ where to compute the periods of $Z$. Since we are integrating over an algebraic cycle (consequently a Hodge cycle) we just care about the $(\frac{n}{2},\frac{n}{2})$-part of $\omega_P$. Thus, we will fix $q=\frac{n}{2}$, and we will work with $\omega_P$ as an element of the quotient $F^{\frac{n}{2}}H^n_\dR(X)/F^{\frac{n}{2}+1}H^n_\dR(X)\simeq H^{\frac{n}{2},\frac{n}{2}}(X)\simeq H^\frac{n}{2}(X,\Omega_X^{\frac{n}{2}})$. After Carlson-Griffiths' work \cite[page 7]{CarlsonGriffiths1980}, we know
\begin{equation}
\label{18.9.2018}    
(\omega_P)^{\frac{n}{2},\frac{n}{2}}=\frac{1}{\frac{n}{2}!}\left\{\frac{P\Omega_J}{F_J}\right\}_{|J|=\frac{n}{2}}\in H^\frac{n}{2}(\U,\Omega_X^\frac{n}{2}).
\end{equation}
Where $\U$ is the \textit{Jacobian covering of $X$}. For $J=(j_0,\ldots,j_\frac{n}{2})$, $F_J:=F_{j_0}\cdots F_{j_\frac{n}{2}},$ where $F_i:=\frac{\partial F}{\partial x_i}$ for every $i=0,\ldots,n+1$, and 
\begin{equation}
\label{22.10.2018}
\Omega_J:=\iota_{\frac{\partial}{\partial x_{j_\frac{n}{2}}}}(\cdots\iota_{\frac{\partial}{\partial x_{j_0}}}(\Omega)\cdots)=(-1)^{j_0+\cdots+j_{\frac{n}{2}}+{\frac{n}{2}+2\choose 2}}\sum_{l=0}^{\frac{n}{2}}(-1)^lx_{k_l}\widehat{dx_{k_l}},    
\end{equation}
for $(k_0,\ldots,k_{\frac{n}{2}-l})$ the multi-index obtained from $(0,1,\ldots,n+1)$ by removing the entries of $J$. We will usually write $(\omega_P)^{\frac{n}{2},\frac{n}{2}}$ in \v{C}ech cohomology as in \eqref{18.9.2018}, but we will denote the period by abuse of notation as $\int_Z \omega_P\in \C$, letting it be understood that we are working under the identifications $F^\frac{n}{2}H^n_\dR(X)/F^{\frac{n}{2}+1}H^n_\dR(X)\simeq H^\frac{n}{2}(\U,\Omega_X^\frac{n}{2})\simeq H^{\frac{n}{2},\frac{n}{2}}(X)\subseteq H^n_\dR(X)$.

\section{Preliminaries on periods}
\label{1}
In this section we prove some preliminary results about periods. We begin by computing periods of top forms over the projective space $\P^{n+1}$. By a top form we mean an element of $H^{n+1,n+1}(\P^{n+1})$ seen as an element of the \v{C}ech cohomology group $H^{n+1}(\U,\Omega_{\P^{n+1}}^{n+1})$ with respect to some affine open cover $\U$ of $\P^{n+1}$.

\begin{prop}[Periods of top forms over the projective space]
\label{08.10.5}
Let $l>0$, and consider a collection of degree $l$ homogeneous polynomials $f_0,\ldots,f_{n+1}\in \C[x_0,\ldots,x_{n+1}]_{l}$, such that $$
\{f_0=\cdots=f_{n+1}=0\}=\varnothing\subseteq \P^{n+1}.
$$
They define the finite morphism $f:\P^{n+1}\rightarrow\P^{n+1}$ given by
$$
f(x_0:\cdots:x_{n+1}):=(f_0:\cdots:f_{n+1}).
$$
Let $\U_f=\{V_i\}_{i=0}^{n+1}$ be the open covering associated to $f$, i.e. $V_i=\{f_i\neq 0\}$. Then the top form 
$$
\frac{\Omega_f}{f_0\cdots f_{n+1}}:=\frac{\sum_{i=0}^{n+1}(-1)^if_i\widehat{df_i}}{f_0\cdots f_{n+1}}\in H^{n+1}(\U_f,\Omega_{\P^{n+1}}^{n+1}),
$$
has period 
$$
\int_{\P^{n+1}}\frac{\Omega_f}{f_0\cdots f_{n+1}}=l^{n+1}\cdot(-1)^{n+2\choose 2}(2\pi\sqrt{-1})^{n+1}.
$$
\end{prop}

\begin{proof} The form $\frac{\Omega}{x_0\cdots x_{n+1}}\in H^{n+1}(\P^{n+1},\Omega_{\P^{n+1}}^{n+1})$ corresponds to a global top form $\omega\in H^{2n+2}_\dR(\P^{n+1})$. We determine this element via the natural isomorphism in hypercohomology $$H^{n+1}(\P^{n+1},\Omega_{\P^{n+1}}^{n+1})\simeq \Hip^{2n+2}(\P^{n+1},\Omega_{\P^{n+1}}^\bullet)\simeq \Hip^{2n+2}(\P^{n+1},\Omega_{(\P^{n+1})^\infty}^\bullet)\simeq H^{2n+2}_\dR(\P^{n+1}),$$
where $\Omega_{(\P^{n+1})^\infty}^k$ denotes the sheaf of $\mathcal{C}^\infty$ differential $k$-forms over $\P^{n+1}$. Let $\{a_i\}_{i=0}^{n+1}$ be a partition of unity subordinated to the standard covering $\{U_i\}_{i=0}^{n+1}$ of $\P^{n+1}$. Computing $\omega$ in terms of this partition of unity, we see that $$\text{Supp }\omega\subseteq U_0\cap\cdots\cap U_{n+1}.$$ In fact, taking the standard coordinates of $U_0$ given by $(z_1,\ldots,z_{n+1})=(\frac{x_1}{x_0},\ldots,\frac{x_{n+1}}{x_0})\in\C^{n+1}$ we can write 
$$
\int_{\P^{n+1}}\omega=(n+1)!(-1)^{n+1}\int_{\C^{n+1}}da_1\wedge\cdots\wedge da_{n+1}\wedge \frac{dz_1}{z_1}\wedge\cdots\wedge\frac{dz_{n+1}}{z_{n+1}}.
$$
Furthermore, we can assume that $a_1,\ldots,a_{n+1}$ are $\mathcal{C}^\infty$ functions defined in $\C^{n+1}$ such that
$$
a_i =
\left\{
	\begin{array}{ll}
		0  & \mbox{if } |z_i|\leq 1 \\
		1 & \mbox{if } |z_i|\geq 2, |z_j|\leq 1 \forall j\in \{1,\ldots,n+1\}\setminus \{i\}
	\end{array}
\right.
$$
and 
$$
a_1+\cdots+a_{n+1}=1\text{ if }\exists j\in\{1,\ldots,n+1\}: |z_j|\geq 2.
$$
Applying Stokes theorem several times we obtain 
$$
\int_{\P^{n+1}}\frac{\Omega}{x_0\cdots x_{n+1}}=(-1)^{n+2\choose 2}\int_{\mathbb{T}^{n+1}}\frac{dz_1}{z_1}\wedge\cdots\wedge\frac{dz_{n+1}}{z_{n+1}}=(-1)^{n+2\choose 2}(2\pi\sqrt{-1})^{n+1}.
$$
Pulling back this form by $f$, it follows that 
$$
\displaystyle\int_{\P^{n+1}}\frac{\Omega_f}{f_0\cdots f_{n+1}}=\deg(f)\cdot\displaystyle\int_{\P^{n+1}}\frac{\Omega}{x_0\cdots x_{n+1}}=\deg(f)\cdot(-1)^{n+2\choose 2}(2\pi\sqrt{-1})^{n+1}.
$$
Since $f$ is defined by a base point free linear system, the fiber of $f$ is generically reduced and corresponds to $l^{n+1}$ points by B\'ezout's theorem.
\end{proof}

\begin{rmk}
The sign appearing in the formula comes from the identification $$H^{n+1}(\P^{n+1},\Omega_{\P^{n+1}}^{n+1})\simeq \Hip^{2n+2}(\P^{n+1},\Omega_{\P^{n+1}}^\bullet)\simeq  H^{2n+2}_\dR(\P^{n+1}).$$ 
We have adopted Carlson and Griffiths' convention for the total complex differential $$D:=(d+(-1)^k\delta)|_{C^{2n+2-k}(\P^{n+1},\Omega_{\P^{n+1}}^k)},$$ associated to the \v{C}ech-de Rham double complex $C^\bullet(\P^{n+1},\Omega_{\P^{n+1}}^\bullet)$, see \cite[page 9]{CarlsonGriffiths1980}. This sign was already pointed out by Deligne in \cite[page 6]{Deligne1982}. The previous proposition can also be found in \cite[Remark (2), page 19]{CarlsonGriffiths1980}.
\end{rmk}

\begin{cor}[Periods of top forms over the projective space II]
\label{15.5.2018}
For every homogeneous polynomial $Q\in \C[x_0,\ldots,x_{n+1}]_{(l-1)(n+2)}$, 
$$
\int_{\P^{n+1}}\frac{Q\Omega}{f_0\cdots f_{n+1}}=c\cdot l^{n+2}\cdot (-1)^{n+2\choose 2}(2\pi\sqrt{-1})^{n+1},
$$
where $c\in \C$ is the unique number such that 
\begin{center}
$Q\equiv c\cdot \det(\Jac(f))$ (mod $\langle f_0,\ldots,f_{n+1}\rangle$).
\end{center}
\end{cor}

\begin{proof}
Using Euler's identity one easily sees that 
\begin{equation}
\label{form}    
\Omega_f=l^{-1}\det(\Jac(f))\Omega,
\end{equation}
where $\Jac(f)=\left(\frac{\partial f_i}{\partial x_j}\right)_{0\le i,j\le n+1}$ is the Jacobian matrix of $f$. 
The rest follows from item (i) of Macaulay's Theorem \ref{08.10.6} and Proposition \ref{08.10.5}.
\end{proof}

\begin{rmk}
Corollary \ref{15.5.2018} implies in particular that the top form $\frac{\Omega}{x_0\cdots x_{n+1}}\in H^{n+1}(\U,\Omega_{\P^{n+1}}^{n+1})$ (with respect to the standard open cover $\U$ of $\P^{n+1}$) integrates $(-1)^{n+2\choose 2}(2\pi\sqrt{-1})^{n+1}$. This can also be deduced from the fact that the polarization $\theta\in H^1(\U,\Omega_{\P^{n+1}}^1)$ is given by $\theta_{ij}=\frac{dx_i}{x_i}-\frac{dx_j}{x_j}$, and so applying several times the twisted product formula we get $$\theta^{n+1}=(-1)^{n+2\choose 2}\frac{\Omega}{x_0\cdots x_{n+1}}.$$
\end{rmk}

\begin{prop}
Under the hypothesis of Proposition \ref{08.10.5}. For every top form 
$$
\omega\in H^{n+1}(\U_f,\Omega_{\P^{n+1}}^{n+1})
$$
there exist explicit polynomials $Q_1,\ldots,Q_k\in \C[x_0,\ldots,x_{n+1}]$ of degree $(l-1)(n+2)$ such that
$$
\int_{\P^{n+1}}\omega=\sum_{i=1}^k\int_{\P^{n+1}}\frac{Q_i\Omega}{f_0\cdots f_{n+1}}.
$$
\end{prop}

\begin{proof}
In general, any element of 
$H^{n+1}(\U_f,\Omega_{\P^{n+1}}^{n+1})$ is of the form
$$
\omega=\frac{P\Omega}{f_0^{\alpha_0}\cdots f_{n+1}^{\alpha_{n+1}}},
$$
where $\alpha_0,\ldots,\alpha_{n+1}\in \Z_{>0}$ with $l\cdot (\alpha_0+\cdots+\alpha_{n+1})=\deg(P)+n+2$.
Using Macaulay's Theorem \ref{08.10.6} applied to $\langle f_0,\ldots,f_{n+1}\rangle\subseteq\C[x_0,\ldots,x_{n+1}]$, we obtain that 
$$
P=\sum_{l(\beta_0+\cdots+\beta_{n+1})=\deg(P)-l(n+2)}f_0^{\beta_0}\cdots f_{n+1}^{\beta_{n+1}}P_\beta,
$$
with $\deg(P_\beta)=(l-1)(n+2)$. This reduces the problem of computing periods of top forms over $\P^{n+1}$ with respect to the cover $\U_f$, to forms 
\begin{equation}
\label{08.10.1}
\frac{P_\beta\Omega}{f_0^{\alpha_0}\cdots f_{n+1}^{\alpha_{n+1}}}\in H^{n+1}(\U_f,\Omega_{\P^{n+1}}^{n+1}),
\end{equation}
with $\alpha_0,\ldots,\alpha_{n+1}\in\Z$ such that $\alpha_0+\cdots+\alpha_{n+1}=n+2$ and $\deg(P_\beta)=(l-1)(n+2)$. If some $\alpha_i$ is non-positive, \eqref{08.10.1} represents an exact top form of $\P^{n+1}$. Therefore, the following are the forms which may have non-trivial periods
$$
\frac{Q\Omega}{f_0\cdots f_{n+1}}\in H^{n+1}(\U_f,\Omega_{\P^{n+1}}^{n+1}),
$$
with $\deg(Q)=(l-1)(n+2)$.
\end{proof}

In order to compute periods of complete intersection algebraic cycles, we will compute periods of smooth hyperplane sections of a given projective 
smooth variety $X$ (by hyperplane section, we mean that in some projective embedding it corresponds to the intersection of a hyperplane with $X$). In fact, for $Y\hookrightarrow X$ a smooth hypersurface given by $\{F=0\}$, we will give an explicit description of the isomorphism 
\begin{eqnarray*}
      H^{n}(Y,\Omega_Y^{n})&\simeq & H^{n+1}(X,\Omega_X^{n+1}), \\
        \omega & \mapsto & \widetilde{\omega}
\end{eqnarray*}
together with the relation between periods, i.e. the number $a\in\C$ such that 
$$
\int_X\widetilde{\omega}=a\int_Y\omega.
$$

For this purpose recall the long exact sequence 
$$
\cdots\rightarrow H_\dR^{k+1}(X)\rightarrow H^{k+1}_\dR(U)\xrightarrow[]{\res} H_\dR^{k}(Y)\xrightarrow[]{\tau} H_\dR^{k+2}(X)\rightarrow \cdots,
$$
induced by \textit{Poincar\'e residue sequence}
$$
0\rightarrow \Omega_X^\bullet\rightarrow \Omega_X^\bullet(\log Y)\xrightarrow[]{\res} j_*\Omega_Y^{\bullet-1}\rightarrow 0.
$$
Since $H^{2n+1}_\dR(U)=H^{2n+2}_\dR(U)=0$, the \textit{coboundary map} is an isomorphism
$$
H^{2n}_\dR(Y)\overset{\tau}{\simeq} H^{2n+2}_\dR(X).
$$
Noting that these vector spaces are one dimensional, and that $\tau$ induces an isomorphism of Hodge structures of weight $(1,1)$ (since it is nothing else than the wedge product with the cohomological class of $Y$ inside $X$, i.e. its first Chern class), we obtain the desired isomorphism 
\begin{equation}
\label{19.10.2018.2}    
H^n(Y,\Omega_Y^n)\overset{\tau}{\simeq}H^{n+1}(X,\Omega_X^{n+1}).
\end{equation}

\begin{prop}[Coboundary map \eqref{19.10.2018.2} and periods]
\label{27.10.1}
Let $X\subseteq \P^N$ be a smooth complete intersection of dimension $n+1$, and $Y\subseteq X$ a smooth hypersurface given by $\{F=0\}\cap X$, for some homogeneous $F\in \C[x_0,\ldots,x_N]_d$. Let $\U$ be an affine open cover of $X$ and let $\omega\in H^n(\U|_Y,\Omega_Y^n)$. Take any $\overline{\omega}\in C^n(\U, \Omega_X^{n+1}(\log Y))$ such that $\res(\overline{\omega})=\omega$. Define
$$
\widetilde\omega:=\delta(\overline{\omega})\in C^{n+1}(\U,\Omega_X^{n+1}),
$$
where $\delta$ is the \v{C}ech differential. Then $\widetilde{\omega}\in H^{n+1}(X,\Omega_X^{n+1})$ and 
\begin{equation}
\label{2.10.2018}
\int_X\widetilde{\omega}=2\pi\sqrt{-1}\int_Y\omega.    
\end{equation}
\end{prop}

\begin{proof}
The map defined in the statement of the proposition is the coboundary map $\tau$, i.e. $\tau(\omega)=\widetilde\omega$. It is known that the long exact sequence associated to the Poincar\'e residue sequence corresponds to the Thom-Gysin sequence and so $\tau$ corresponds to $\widetilde{\omega}=\omega\wedge[Y]$. Therefore \eqref{2.10.2018} corresponds to Poincar\'e duality.  
\end{proof}

\section{Proof of Theorem \ref{15.5.2018.2}}
\label{2}
Let $X\subseteq \mathbb{P}^{n+1}$ be a smooth degree $d$ hypersurface of even dimension $n$. Given the complete intersection $Z\subseteq X$ of dimension $\frac{n}{2}$, we construct 
a chain of subvarieties 
$$
Z=Z_0\subseteq Z_1\subseteq Z_2\subseteq \cdots\subseteq Z_{\frac{n}{2}+1}=\P^{n+1},
$$
where each $Z_i$ is the intersection of $Z_{i+1}$ with a very ample divisor of $\P^{n+1}$. In order to prove Theorem \ref{15.5.2018.2}, we will apply inductively the coboundary map, to reduce the computation of the period of $Z$ to the computation 
of a period of $\P^{n+1}$.

\begin{prop}
\label{cont}
Both sides of the periods equation \eqref{19.10.2018} depend continuously on the parameters 
$
(f_1,g_1,\ldots,f_{\frac{n}{2}+1},g_{\frac{n}{2}+1})\in \bigoplus_{i=1}^{\frac{n}{2}+1}\C[x_0,\ldots,x_{n+1}]_{d_i}\oplus\C[x_0,\ldots,x_{n+1}]_{d-d_i},
$ 
such that $F:=f_1g_1+\cdots+f_{\frac{n}{2}+1}g_{\frac{n}{2}+1}$.
\end{prop}

\begin{proof}
Consider 
$$
U:=\left\{(f_1,g_1,\ldots,f_{\frac{n}{2}+1},g_{\frac{n}{2}+1})\in \bigoplus_{i=1}^{\frac{n}{2}+1}\C[x]_{d_i}\oplus\C[x]_{d-d_i}:\right.
$$
$$
X:=\{f_1g_1+\cdots+f_{\frac{n}{2}+1}g_{\frac{n}{2}+1}=0\}\text{ is smooth and }
$$
$$
\left. Z:=\{f_1=f_2=\cdots =f_{\frac{n}{2}+1}=0\}\text{ is a complete intersection}\right\}.
$$
Let $\sigma:=(d-2)(\frac{n}{2}+1)$ and fix any $P\in\C[x]_\sigma$. For $(f_1,g_1,\ldots,f_{\frac{n}{2}+1},g_{\frac{n}{2}+1})\in U$, we know that the Jacobian ideal $J^F:=\langle \frac{\partial F}{\partial x_0},\ldots,\frac{\partial F}{\partial x_{n+1}}\rangle$ (where $F:=f_1g_1+\cdots+f_{\frac{n}{2}+1}g_{\frac{n}{2}+1}\in \C[x]_d$) is Artinian Gorenstein of $soc(J^F)=2\sigma$, and that $\det(\Hess(F))\in \C[x]_{2\sigma}\setminus J^F_{2\sigma}$ (by Corollary \ref{15.5.2018}). Therefore there exists a unique number $c\in\C$ such that 
$$
P\cdot \det(\Jac(f_1,g_1,\ldots,f_{\frac{n}{2}+1},g_{\frac{n}{2}+1}))\equiv c\cdot \det(\Hess(F))\hspace{5mm}(\text{mod }J^F).
$$
We claim that this number $c$ depends continuously on 
$$
\lambda:=(f_1,g_1,\ldots,f_{\frac{n}{2}+1},g_{\frac{n}{2}+1})\in U.
$$
In fact, consider the $\C$-vector space $V:=\C[x]_{2\sigma}$. For every $\lambda\in U$ define the hyperplane $V_\lambda:=J^F_{2\sigma}\subseteq V$, we claim that $V_\lambda$ varies continuously with respect to $\lambda$ in the space of hyperplanes of $V$, in fact, each $V_\lambda$ is generated as $\C$-vector space by the vectors 
$$
V_\lambda=\left\langle \frac{\partial F_\lambda}{\partial x_i}x^{I}: i=0,\ldots,n+1, x^I\text{ monomials with }|I|=2\sigma-d+1\right\rangle,
$$
where $F_\lambda:=F=f_1g_1+\cdots+f_{\frac{n}{2}+1}g_{\frac{n}{2}+1}$, and each of these vectors depend continuously on $\lambda\in U$ (here we are using the non-trivial fact that we know a priori that the generated spaces are hyperplanes). In consequence, there exists a continuous map 
$$
\varphi:U\rightarrow \P(V^*)
$$ 
such that $V_\lambda=\text{Ker }\varphi_\lambda$. Now we can compute $c$ in terms of continuous functions depending on $\lambda\in U$ as
$$
c=\frac{\varphi_\lambda(P\cdot \det(\Jac(\lambda)))}{\varphi_\lambda(\det(\Hess(F_\lambda)))}.
$$
\end{proof}

Proposition \ref{cont} implies that it is enough to prove Theorem \ref{15.5.2018.2} for a general $(f_1,g_1,\ldots,$ $f_{\frac{n}{2}+1},g_{\frac{n}{2}+1})$. This is why we may assume each $Z_{l-1}$ is a smooth hyperplane section of $Z_{l}$, for $l=1,\ldots,\frac{n}{2}+1$, as in the hypothesis of Proposition \ref{27.10.1}.

\bigskip

Let $P\in \C[x_0,\ldots,x_{n+1}]_{(d-2)(\frac{n}{2}+1)}$, $\U$ be the Jacobian cover of $\P^{n+1}$, and 
$$
\omega:=\omega_P\in H^\frac{n}{2}(\U|_X,\Omega_X^\frac{n}{2}),
$$
as in \eqref{18.9.2018}. Using Proposition \ref{27.10.1} we construct inductively
\begin{center}
${\omega}^{(0)}:=\omega|_Z\in H^{\frac{n}{2}}(\U|_Z,\Omega_{Z}^{\frac{n}{2}})$ and $Z_0:=Z$.
\end{center}
Then for $l=1,\ldots,\frac{n}{2}+1$ we define
\begin{center}
${\omega}^{(l)}:=\widetilde{{\omega}^{(l-1)}}\in H^{\frac{n}{2}+l}(\U|_{Z_l},\Omega_{Z_l}^{\frac{n}{2}+l})$ and $Z_l:=\{f_{l+1}=\cdots=f_{\frac{n}{2}+1}=0\}\subseteq\mathbb{P}^{n+1}$.
\end{center}
Observe that $Z_{\frac{n}{2}+1}=\mathbb{P}^{n+1}$. 

\begin{lemma}
\label{15.5.2018.3}
For $l\in\{0,\ldots,\frac{n}{2}+1\}$ and $J=(j_0,\ldots,j_{\frac{n}{2}+l})$, $0\le j_0<\cdots<j_{\frac{n}{2}+l}\le n+1$,
$$
\left.\begin{array}{llll}
({\omega}^{(l)})_J&=\dfrac{(-1)^{{\frac{n}{2}+2 \choose 2}+j_0+\cdots+j_{\frac{n}{2}+l}}Pd^ld_1\cdots d_l}{\frac{n}{2}!\cdot F_J}\\

&\cdot\left[\displaystyle\sum_{m=1}^l(-1)^{m-1}g_m\widehat{\frac{dg_m}{d}}\bigwedge_{r=0}^{\frac{n}{2}-l}dx_{k_r}\bigwedge_{t=1}^l\frac{df_t}{d_t}\right. \\

&+(-1)^l\displaystyle\sum_{p=0}^{\frac{n}{2}-l}(-1)^px_{k_p}\bigwedge_{s=1}^l\frac{dg_s}{d}\wedge\widehat{dx_{k_p}}\wedge\bigwedge_{t=1}^l\frac{df_t}{d_t}\\

&\left.+(-1)^{\frac{n}{2}+l}\displaystyle\sum_{q=1}^l\widehat{\frac{dg_q}{d}}\wedge\frac{dF}{d}\bigwedge_{r=0}^{\frac{n}{2}-l}dx_{k_r}\wedge\widehat{\frac{df_q}{d_q}}\right],
       \end{array}
\right.
$$
where $K=(k_0,\ldots,k_{\frac{n}{2}-l})$ is obtained from $(0,1,\ldots,n+1)$ by removing the entries of $J$ (the notation $\widehat{\frac{dg_m}{d}}:=\frac{dg_1}{d}\wedge\cdots \widehat{\frac{dg_m}{d}}\cdots \wedge \frac{dg_l}{d}$, and analogously for $\widehat{dx_{k_p}}$ and $\widehat{\frac{df_q}{d_q}}$, was already set in Notation \ref{22.10.2018.2}). 
\end{lemma}

\begin{proof} We proceed by induction on $l$:\\
\\
Computing $\Omega_J$ (as in \eqref{22.10.2018}) we get
$$
({\omega}^{(0)})_{j_0\cdots j_{\frac{n}{2}}}=(\omega)_{j_0\cdots j_{\frac{n}{2}}}=\frac{(-1)^{{\frac{n}{2}+2 \choose 2}+j_0+\cdots+j_\frac{n}{2}}P}{\frac{n}{2}!\cdot F_J}\left[\displaystyle\sum_{p=0}^{\frac{n}{2}}(-1)^px_{k_p}\widehat{dx_{k_p}}\right].
$$
Assuming it is true for $l$, then we can take $\overline{\omega^{(l)}}_J\in C^{\frac{n}{2}+l}(Z_{l+1},\Omega_{Z_{l+1}}^{\frac{n}{2}+l+1}(\log Z_l))$ given by
$$\left.\begin{array}{lllll}
\overline{\omega^{(l)}}_J&=\dfrac{(-1)^{{\frac{n}{2}+2 \choose 2}+j_0+\cdots+j_{\frac{n}{2}+l}}Pd^ld_1\cdots d_{l+1}}{\frac{n}{2}!\cdot F_J\cdot f_{l+1}}\\

&\cdot\left[\displaystyle\sum_{m=1}^l(-1)^{m-1}g_m\widehat{\frac{dg_m}{d}}\bigwedge_{r=0}^{\frac{n}{2}-l}dx_{k_r}\bigwedge_{t=1}^{l+1}\frac{df_t}{d_t}\right. \\

&\left.+(-1)^l\displaystyle\sum_{p=0}^{\frac{n}{2}-l}(-1)^px_{k_p}\bigwedge_{s=1}^l\frac{dg_s}{d}\wedge\widehat{dx_{k_p}}\wedge\bigwedge_{t=1}^{l+1}\frac{df_t}{d_t}\right.\\

&\left.+(-1)^{\frac{n}{2}+l}\displaystyle\sum_{q=1}^l\widehat{\frac{dg_q}{d}}\wedge\frac{dF}{d}\bigwedge_{r=0}^{\frac{n}{2}-l}dx_{k_r}\wedge\widehat{\frac{df_q}{d_q}}\wedge \frac{df_{l+1}}{d_{l+1}}\right. \\

&\left.+(-1)^{\frac{n}{2}+l+1}f_{l+1}\displaystyle\sum_{u=1}^{l+1}\widehat{\frac{dg_u}{d}}\bigwedge_{r=0}^{\frac{n}{2}-l}dx_{k_r}\wedge\widehat{\frac{df_u}{d_u}}\right].
       \end{array}
\right.$$
Applying the \v{C}ech differential $\delta$ we get
$$\begin{array}{ll}
\omega^{(l+1)}_J&=\dfrac{(-1)^{{\frac{n}{2}+2 \choose 2}+j_0+\cdots+j_{\frac{n}{2}+l+1}}Pd^ld_1\cdots d_{l+1}}{\frac{n}{2}!\cdot F_J\cdot f_{l+1}}\\

&\cdot\left[\displaystyle\sum_{m=1}^l(-1)^{m-1}g_m\widehat{\frac{dg_m}{d}}\wedge\left(\sum_{p=0}^{\frac{n}{2}+l+1}F_{j_p}dx_{j_p}\right)\bigwedge_{r=0}^{\frac{n}{2}-l-1}dx_{k_r}\bigwedge_{t=1}^{l+1}\frac{df_t}{d_t}\right. \\

&+(-1)^{l}\left(\displaystyle\sum_{p=0}^{\frac{n}{2}+l+1}F_{j_p}x_{j_p}\right)\displaystyle\bigwedge_{s=1}^l\dfrac{dg_s}{d}\displaystyle\bigwedge_{q=0}^{\frac{n}{2}-l-1}dx_{k_q}\displaystyle\bigwedge_{t=1}^{l+1}\dfrac{df_t}{d_t}\\

&+(-1)^{l+1}\displaystyle\sum_{p=0}^{\frac{n}{2}-l-1}(-1)^px_{k_p}\bigwedge_{s=1}^l\frac{dg_s}{d}\wedge\left(\sum_{r=0}^{\frac{n}{2}+l+1}F_{j_r}dx_{j_r}\right)\wedge\widehat{dx_{k_p}}\bigwedge_{t=1}^{l+1}\frac{df_t}{d_t}\\
\end{array}
$$

$$\begin{array}{ll}

&+(-1)^{\frac{n}{2}+l}\displaystyle\sum_{q=1}^l\widehat{\frac{dg_q}{d}}\wedge\frac{dF}{d}\wedge\left(\sum_{p=0}^{\frac{n}{2}+l+1}F_{j_p}dx_{j_p}\right)\bigwedge_{r=0}^{\frac{n}{2}-l-1}dx_{k_r}\wedge\widehat{\frac{df_q}{d_q}}\wedge \frac{df_{l+1}}{d_{l+1}} \\

&\left.+(-1)^{\frac{n}{2}+l+1}f_{l+1}\displaystyle\sum_{u=1}^{l+1}\widehat{\frac{dg_u}{d}}\wedge\left(\sum_{p=0}^{\frac{n}{2}+l+1}F_{j_p}dx_{j_p}\right)\bigwedge_{r=0}^{\frac{n}{2}-l-1}dx_{k_r}\wedge\widehat{\frac{df_u}{d_u}}\right]
\end{array}$$
$$\begin{array}{llll}
&=\dfrac{(-1)^{{\frac{n}{2}+2 \choose 2}+j_0+\cdots+j_{\frac{n}{2}+l+1}}Pd^{l+1}d_1\cdots d_{l+1}}{\frac{n}{2}!\cdot F_J\cdot f_{l+1}}\\

&\cdot\left[\displaystyle\sum_{m=1}^l(-1)^{m-1}g_m\widehat{\frac{dg_m}{d}}\wedge \frac{dF}{d}\bigwedge_{r=0}^{\frac{n}{2}-l-1}dx_{k_r}\bigwedge_{t=1}^{l+1}\frac{df_t}{d_t}\right. \\

&+(-1)^{l}F\displaystyle\bigwedge_{s=1}^l\dfrac{dg_s}{d}\displaystyle\bigwedge_{q=0}^{\frac{n}{2}-l-1}dx_{k_q}\displaystyle\bigwedge_{t=1}^{l+1}\dfrac{df_t}{d_t}\\

&+(-1)^{l+1}\displaystyle\sum_{p=0}^{\frac{n}{2}-l-1}(-1)^px_{k_p}\bigwedge_{s=1}^l\frac{dg_s}{d}\wedge\frac{dF}{d}\wedge\widehat{dx_{k_p}}\bigwedge_{t=1}^{l+1}\frac{df_t}{d_t}\\

&+(-1)^{\frac{n}{2}+l}\displaystyle\sum_{q=1}^l\widehat{\frac{dg_q}{d}}\wedge\frac{dF}{d}\wedge\frac{dF}{d}\bigwedge_{r=0}^{\frac{n}{2}-l-1}dx_{k_r}\wedge\widehat{\frac{df_q}{d_q}}\wedge \frac{df_{l+1}}{d_{l+1}} \\

&\left.+(-1)^{\frac{n}{2}+l+1}f_{l+1}\displaystyle\sum_{u=1}^{l+1}\widehat{\frac{dg_u}{d}}\wedge\frac{dF}{d}\bigwedge_{r=0}^{\frac{n}{2}-l-1}dx_{k_r}\wedge\widehat{\frac{df_u}{d_u}}\right].
\end{array}$$
Replacing $F=f_1g_1+\cdots+f_{\frac{n}{2}+1}g_{\frac{n}{2}+1}$ in the first three sums above we obtain the claimed equality.
\end{proof}

\noindent\textbf{Proof of Theorem \ref{15.5.2018.2}} Let $P\in\C[x_0,\ldots,x_{n+1}]$ be an homogeneous polynomial of degree $\sigma=(d-2)(\frac{n}{2}+1)$, and let
$$
\omega=\omega_P=\res\left(\frac{P\Omega}{F^{\frac{n}{2}+1}}\right).
$$
In order to compute the period of $\omega$ over the complete intersection cycle
$$
Z=\{f_1=\cdots=f_{\frac{n}{2}+1}=0\}\subseteq\P^{n+1},
$$
we apply Proposition \ref{27.10.1} several times. Recall that by Proposition \ref{cont} we can reduce ourselves to a general choice of polynomials $f_1,\ldots,f_{\frac{n}{2}+1}$, and so we can assume each
$$
Z_l:=\{f_{l+1}=\cdots=f_{\frac{n}{2}+1}=0\}\subseteq\P^{n+1}
$$
is a smooth hypersurface of $Z_{l+1}$, for each $l=0,\ldots,\frac{n}{2}$. The result of this iterative application of Proposition \ref{27.10.1} was computed in Lemma \ref{15.5.2018.3}. It follows that for $l=\frac{n}{2}+1$ we have
$$
\left.\begin{array}{ll}
({\omega}^{(\frac{n}{2}+1)})_{0\cdots n+1}&=\dfrac{(-1)^{\frac{n}{2}+1\choose 2}Pd^{\frac{n}{2}+1}d_1\cdots d_{\frac{n}{2}+1}}{\frac{n}{2}!\cdot F_0\cdots F_{n+1}}\left[
\displaystyle\sum_{m=1}^{\frac{n}{2}+1}(-1)^{m-1}g_m\widehat{\frac{dg_m}{d}}\bigwedge_{t=1}^{\frac{n}{2}+1}\frac{df_t}{d_t}\right.\\

&\left.+
(-1)^{n+1}\displaystyle\sum_{q=1}^{\frac{n}{2}+1}\widehat{\frac{dg_q}{d}}\wedge\frac{dF}{d}\wedge\widehat{\frac{df_q}{d_q}}\right].
\end{array}
\right.
$$
Replacing $F=f_1g_1+\cdots+f_{\frac{n}{2}+1}g_{\frac{n}{2}+1}$ on the above equation we obtain
$$
\left.\begin{array}{lll}
({\omega}^{(\frac{n}{2}+1)})_{0\cdots n+1}&=\dfrac{(-1)^{\frac{n}{2}+1\choose 2}Pd^{\frac{n}{2}+1}d_1\cdots d_{\frac{n}{2}+1}}{\frac{n}{2}!\cdot F_0\cdots F_{n+1}}\\

&\cdot\left[
\displaystyle\sum_{m=1}^{\frac{n}{2}+1}(-1)^{m-1}\left(\frac{d-d_m}{d}\right)g_m\widehat{\frac{dg_m}{d}}\bigwedge_{t=1}^{\frac{n}{2}+1}\frac{df_t}{d_t}\right.\\

&\left.+
(-1)^{\frac{n}{2}}\displaystyle\sum_{q=1}^{\frac{n}{2}+1}(-1)^qf_q\bigwedge_{s=1}^{\frac{n}{2}+1}\frac{dg_s}{d}\wedge\widehat{\frac{df_q}{d_q}}\right].\\

&=\dfrac{(-1)^{{n+2\choose 2}}Pe_0\cdots e_{n+1}}{\frac{n}{2}!\cdot F_0\cdots F_{n+1}}
\displaystyle\sum_{k=0}^{n+1}(-1)^kh_k\widehat{\frac{dh_k}{e_k}},
\end{array}
\right.
$$
where $e_k=\deg(h_k)$. Replacing $e_ih_i=\displaystyle\sum_{j=0}^{n+1}\frac{\partial h_i}{\partial x_j}\cdot x_j$ and $dh_i=\displaystyle\sum_{j=0}^{n+1}\frac{\partial h_i}{\partial x_j}dx_j$ 
we get
$$
({\omega}^{(\frac{n}{2}+1)})_{0\cdots n+1}=\dfrac{(-1)^{\frac{n}{2}+1}P\cdot  \det(\Jac(H))}{\frac{n}{2}!\cdot F_0\cdots F_{n+1}}
\displaystyle\sum_{k=0}^{n+1}(-1)^kx_k\widehat{dx_k}.
$$
Corollary \ref{15.5.2018} tells us what is the period of $\omega^{(\frac{n}{2}+1)}$ above, and Proposition \ref{27.10.1} tells how to obtain the period of $\omega$ from this period. Putting all together we get the desired result.
\qed

\section{Proof of Theorem \ref{cycleclass}}
\label{pt1}

After Griffiths basis theorem we know that 
\begin{equation}
\label{zcycl}    
[Z]=(\omega_{P_Z})^{\frac{n}{2},\frac{n}{2}}+\alpha\theta^\frac{n}{2}\in H^{\frac{n}{2},\frac{n}{2}}(X)
\end{equation}
for some $\alpha\in\C$ and some $P_Z\in\C[x_0,\ldots,x_{n+1}]_{(d-2)(\frac{n}{2}+1)}$. In order to compute $\alpha$ let us integrate the polarization $\theta^\frac{n}{2}$ over $Z$
$$
\deg(Z)=\frac{1}{(2\pi\sqrt{-1})^\frac{n}{2}}\int_Z\theta^\frac{n}{2}=\frac{1}{(2\pi\sqrt{-1})^{n}}\int_X \theta^\frac{n}{2}\wedge \alpha\theta^\frac{n}{2}=\alpha\cdot \deg(X),
$$
and so $\alpha=\frac{\deg(Z)}{\deg(X)}$. We will need the following fact whose proof was essentially done in the proof of \cite[Theorem 2]{CarlsonGriffiths1980}.

\begin{prop}
\label{wedge}
Let $X\subseteq\P^{n+1}$ be a smooth degree $d$ hypersurface of even dimension $n$. Let $P,Q\in\C[x_0,\ldots,x_{n+1}]_{(d-2)(\frac{n}{2}+1)}$, then
$$
\int_X \omega_P\wedge\omega_Q=\frac{-(2\pi\sqrt{-1})^n}{(\frac{n}{2}!)^2}c\cdot (d-1)^{n+2}d,
$$
where $c\in \C$ is the unique number such that 
$$
PQ\equiv c\cdot \det(\Hess(F))\hspace{3mm}(\text{mod }J^F).
$$
\end{prop}

\begin{proof}
Let $\U$ be the Jacobian covering of $\P^{n+1}$. By \eqref{18.9.2018} we know explicitly how $(\omega_P)^{\frac{n}{2},\frac{n}{2}}$ and $(\omega_Q)^{\frac{n}{2},\frac{n}{2}}$ look like in the \v{C}ech cohomology group $H^\frac{n}{2}(\U,\Omega_X^\frac{n}{2})$. Then we can also compute $(\omega_P\wedge\omega_Q)^{n,n}\in H^n(\U,\Omega_X^n)$ by performing the twisted product
$$
((\omega_P)^{\frac{n}{2},\frac{n}{2}}\wedge(\omega_Q)^{\frac{n}{2},\frac{n}{2}})_{0\cdots\widehat{m}\cdots n+1}=\left\{
	\begin{array}{ll}
		\frac{(-1)^{\frac{n}{2}+m}PQx_m\Omega_{(\frac{n}{2}+1)}F_m}{(\frac{n}{2}!)^2F_0\cdots F_{n+1}\cdot F_{\frac{n}{2}+1}}  & \mbox{if } m\le \frac{n}{2} \\
		\frac{(-1)^{\frac{n}{2}+m}PQx_m\Omega_{(\frac{n}{2})}F_m}{(\frac{n}{2}!)^2F_0\cdots F_{n+1}\cdot F_{\frac{n}{2}}} & \mbox{if } m>\frac{n}{2}
	\end{array}
\right.
$$
where $\Omega_{(i)}=\iota_{\frac{\partial}{\partial x_{i}}}(\Omega)$ for $i=\frac{n}{2},\frac{n}{2}+1$. A direct application of Proposition \ref{27.10.1} gives us 
$$
\int_{\P^{n+1}} \widetilde{\omega}=2\pi\sqrt{-1}\int_X \omega_P\wedge\omega_Q,
$$
for 
$$
\widetilde{\omega}=\frac{d\cdot(-1)^{\frac{n}{2}}PQ\Omega}{(\frac{n}{2}!)^2F_0\cdots F_{n+1}}\in C^{n+1}(\U,\Omega_{\P^{n+1}}^{n+1}).
$$
The result follows from Corollary \ref{15.5.2018}.
\end{proof}

\noindent{\textbf{Proof of Theorem \ref{cycleclass}}} Let $R_Z:=\frac{-\frac{n}{2}!}{\deg(X)}\det(\Jac(H))\in \C[x_0,\ldots,x_{n+1}]_{(d-2)(\frac{n}{2}+1)}$, we claim that $P_Z=R_Z$ (where $P_Z$ is given by \eqref{zcycl}). In fact, since the wedge product on $H^n_\dR(X)_{\prim}$ is not degenerated it is enough to check that
$$
\frac{1}{(2\pi\sqrt{-1})^\frac{n}{2}}\int_Z\omega_P=\frac{1}{(2\pi\sqrt{-1})^n}\int_X\omega_{P_Z}\wedge\omega_P=\frac{1}{(2\pi\sqrt{-1})^n}\int_X\omega_{R_Z}\wedge\omega_P, 
$$
$\forall P\in\C[x_0,\ldots,x_{n+1}]_{(d-2)(\frac{n}{2}+1)}$, which follows from Theorem \ref{15.5.2018.2} and Proposition \ref{wedge}.
\qed

\section{Hodge locus}
\label{hodgelocus}
Before going to the applications of Theorem \ref{cycleclass} and Theorem \ref{15.5.2018.2}, let us recall the Hodge locus associated to a Hodge cycle inside a smooth degree $d$ hypersurface of the projective space $\P^{n+1}$, of even dimension $n$.

\begin{dfn}
\label{5.11.2018}
Let $\pi:X\rightarrow T$ be the family of smooth degree $d$ hypersurfaces of $\P^{n+1}$, of even dimension $n$. Fix a parameter $0\in T$, and a Hodge cycle $\lambda_0\in H^n(X_0,\Z)\cap H^{\frac{n}{2},\frac{n}{2}}(X_0)$. Since $\pi$ is a locally trivial fibration, we can extend $\lambda_0$ to a polydisc around $0\in T$ by parallel transport. If we denote this extension by $\lambda_t\in H^n(X_t,\Z)$, the \textit{Hodge locus associated to $\lambda_0$} is
$$
V_{\lambda_0}:=\{t\in (T,0): \lambda_t\in H^n(X_t,\Z)\cap H^{\frac{n}{2},\frac{n}{2}}(X_t) \},
$$
where $(T,0)$ denotes the germ of neighbourhoods of $0\in T$ in the analytic topology. Considering $\omega_1,\ldots,\omega_k\in H^n_\dR(X/T)$ such that they form a basis for $F^{\frac{n}{2}+1}H^n_\dR(X_t)$ for every $t$ in a neighbourhood of $0\in T$, we can induce an structure of analytic space in the Hodge locus as
$$
\mathcal{O}_{V_{\lambda_0}}=\frac{\mathcal{O}_{(T,0)}}{\langle f_1,\ldots,f_k\rangle},
$$
where $\lambda_t=\sum_{i=1}^kf_i(t)\omega_i(t)$ for every $t\in(T,0)$. This structure might be non-reduced, see for instance \cite[page 154, Exercise 2]{vo03}.
\end{dfn}

We will end this section with a restatement of a well known fact relating periods of a Hodge cycle, to the Zariski tangent space of its associated Hodge locus.

\begin{prop}
\label{17.10.2018}
Let $T\subseteq\C[x_0,\ldots,x_{n+1}]_d$ be the parameter space of smooth degree $d$ hypersurfaces of $\P^{n+1}$, of even dimension $n$. For $t\in T$, let $X_t=\{F=0\}\subseteq \P^{n+1}$ be the corresponding hypersurface. For every Hodge cycle $\lambda\in H^n(X_t,\Z)\cap H^{\frac{n}{2},\frac{n}{2}}(X_t)$, we can compute the Zariski tangent space of its associated Hodge locus $V_\lambda$ as
$$
T_tV_\lambda=\left\{P\in \C[x_0,\ldots,x_{n+1}]_d: \int_{\delta}\res\left(\frac{PQ\Omega}{F^{\frac{n}{2}+1}}\right)=0, \forall Q\in \C[x_0,\ldots,x_{n+1}]_{d\frac{n}{2}-n-2}\right\},
$$
where $\delta\in H_n(X_t,\Z)$ is the dual of $\lambda\in H^n(X_t,\Z)$ (note that $H_n(X_t,\Z)$ is free).
\end{prop}

\begin{proof}
We know from Voisin \cite[Lemma 5.16]{vo03}, that 
$$
T_tV_\lambda=\text{Ker }\overline{\nabla}_t(\lambda_t),
$$
where $\overline{\nabla}_t$ is induced by the infinitesimal variations of Hodge structures. This map is well known in the case of hypersurfaces and corresponds with 
$$
\overline{\nabla}_t: H^{\frac{n}{2},\frac{n}{2}}(X_t)\times T_tT\rightarrow H^{\frac{n}{2}-1,\frac{n}{2}+1}(X_t)^*,
$$
given by the multiplication map (see \cite[Theorem 6.17]{vo03})
$$
(\overline{\nabla}_t(\lambda_t,P))(\res\left(\frac{Q\Omega}{F^\frac{n}{2}}\right))=\int_\delta \res\left(\frac{PQ\Omega}{F^{\frac{n}{2}+1}}\right).
$$
Note that we have identified $P\in \C[x_0,\ldots,x_{n+1}]_d\simeq T_tT$.
\end{proof}

\section{First applications}
\label{3}
\begin{dfn}
\label{4.8.2018}
We will say that an algebraic cycle $\delta\in \CH^n(X)$ is of \textit{complete intersection type} if 
$$
\delta=\sum_{i=1}^kn_i\cdot Z_i,
$$
for $Z_1,\ldots,Z_k\subseteq X$ a set of $\frac{n}{2}$-dimensional subvarieties that are complete intersection inside $\P^{n+1}$, given by
$$
Z_i=\{f_{i,1}=\cdots=f_{i,\frac{n}{2}+1}=0\},
$$
for every $i=1,\ldots,k$, such that there exist $g_{i,1},\ldots,g_{i,k}\in \C[x_0,\ldots,x_{n+1}]$ with
$$
F=\sum_{j=1}^{\frac{n}{2}+1}f_{i,j}g_{i,j}.
$$
We denote this subspace by $\CH^n(X)_{cit}$. For every $\delta\in \CH^n(X)_{cit}$, we define its \textit{associated polynomial}
$$
P_\delta:=\sum_{i=1}^k n_i\cdot \det(\Jac(H_i))\in R^F_{(d-2)(\frac{n}{2}+1)},
$$
where $H_i:=(f_{i,1},g_{i,1},\ldots,f_{i,\frac{n}{2}+1},g_{i,\frac{n}{2}+1})$. We define its \textit{degree} as its degree as an element of $H_n(\P^{n+1},\Z)$, i.e. $\deg(\delta):=\sum_{i=1}^k n_i\cdot \deg(Z_i)$. It follows from Theorem \ref{cycleclass} and the linearity of the cycle class map that
\begin{equation}
\label{cyclcit}
[\delta]=\frac{\deg(\delta)}{\deg(X)}\theta^\frac{n}{2}-\frac{\frac{n}{2}!}{\deg(X)}(\omega_{P_\delta})^{\frac{n}{2},\frac{n}{2}}.
\end{equation}
\end{dfn}

\begin{cor}
\label{4.8.2018.3}
Let $X\subseteq \mathbb{P}^{n+1}$ be a smooth hypersurface given by $$X=\{F=0\}.$$ If $\delta,\mu\in \CH^n(X)_{cit}$ are complete intersection type algebraic cycles, then
\begin{itemize}
    \item[(i)] $P_\delta\in J^F$ if and only if $[\delta]=\alpha\cdot[X\cap \P^{\frac{n}{2}+1}]$, for $\alpha=\deg(\delta)/\deg(X)$.
    \item[(ii)] Let $c\in\C$ be the unique number such that $P_\delta\cdot P_\mu\equiv c\cdot \det(\Hess(F))$ (mod $J^F$), then
    \begin{equation}
    \label{intform}
    \delta\cdot \mu=\frac{\deg(\delta)\cdot \deg(\mu)}{\deg(X)}-c\cdot\frac{(\deg(X)-1)^{n+2}}{\deg(X)}.
    \end{equation}
\end{itemize}
\end{cor}

\begin{proof}
The first part is a direct application of Griffiths basis theorem and \eqref{cyclcit}. The second part is a direct application of the fact
$$
\delta\cdot\mu=\frac{1}{(2\pi\sqrt{-1})^n}\int_X[\delta]\wedge[\mu],
$$
together with equation \eqref{cyclcit}, Corollary \ref{15.5.2018} and Proposition \ref{wedge}.
\end{proof}


\begin{rmk}
It follows from \eqref{intform} that for every pair of algebraic cycles $\delta,\mu\in \CH^\frac{n}{2}(X)$, the unique number $c\in\C$ such that $P_\delta\cdot P_\mu\equiv c\cdot \det(\Hess(F))$ (mod $J^F$) is in fact a rational number such that
\begin{center}
$c\cdot (d-1)^{n+2}\in\Z$ and $c\cdot (d-1)^{n+2}\equiv \deg(\delta)\cdot \deg(\mu)$ (mod $d$).
\end{center}
In general, it is not known how to determine whether a given element of Griffiths basis $(\omega_P)^{\frac{n}{2},\frac{n}{2}}\in H^{\frac{n}{2},\frac{n}{2}}(X)$ is an integral or rational class in terms of the polynomial $P\in \C[x_0,\ldots,x_{n+1}]_{(d-2)(\frac{n}{2}+1)}$. Equation \eqref{intform} gives us a (computable) necessary condition: If $(\omega_P)^{\frac{n}{2},\frac{n}{2}}\in H^{\frac{n}{2},\frac{n}{2}}(X)\cap H^n(X,\Z)$ then for every complete intersection type algebraic cycle $\delta\in \CH^\frac{n}{2}(X)_{cit}$
\begin{equation}
P\cdot P_\delta\equiv c\cdot \det(\Hess(F)) \hspace{5mm}(\text{mod }J^F),
\end{equation}
for some $c\in \Q$ such that $c\cdot (d-1)^{n+2}\in\frac{n}{2}!\Z$. A further condition that follows from Proposition \ref{wedge} is 
\begin{equation}
P^2\equiv c_P\cdot \det(\Hess(F))\hspace{5mm}(\text{mod }J^F),
\end{equation}
for some $c_P\in \Q$ such that $c_P (d-1)^{n+2}d\in(\frac{n}{2}!)^2\Z$.
\end{rmk}

\begin{rmk}
Another observation we can derive from Theorem \ref{15.5.2018.2} is that each period is of the form $(2\pi\sqrt{-1})^{\frac{n}{2}}$ times an element from a number field $k$, where $k$ is the smallest number field containing the coefficients of $f_1,g_1,\ldots,f_{\frac{n}{2}+1},g_{\frac{n}{2}+1}$, i.e. the periods belong to the same field where we can decompose $F$ as $f_1g_1+\cdots+f_{\frac{n}{2}+1}g_{\frac{n}{2}+1}$. This was already mentioned in Deligne's work about absolute Hodge cycles (see \cite[Proposition 7.1]{Deligne1982}).
\end{rmk}

One of the main ingredients of the proof of Theorem \ref{11.10.2018} is the description of the Zariski tangent space of the Hodge locus $V_{[\delta]}$ as the degree $d$ part of the quotient ideal $(J^F:P_\delta)$. 

\begin{cor}
\label{16.10.2018}
Let $T$ be the parameter space of smooth degree $d$ hypersurfaces of $\P^{n+1}$, of even dimension $n$. For $t\in T$, let $X_t=\{F=0\}\subseteq\P^{n+1}$ be the corresponding hypersurface. If $\delta\in \CH^n(X_t)_{cit}$ is a complete intersection type algebraic cycle, then 
$$
T_tV_{[\delta]}=(J^F:P_\delta)_d.
$$
\end{cor}

\begin{proof}
By Proposition \ref{17.10.2018} and Theorem \ref{15.5.2018.2} we have
$$
T_tV_{[\delta]}=\{P\in\C[x_0,\ldots,x_{n+1}]_d: P\cdot Q\cdot P_\delta\in J^F, \forall Q\in \C[x_0,\ldots,x_{n+1}]_{d\frac{n}{2}-n-2}\}.
$$
By item (ii) of Macaulay's Theorem \ref{08.10.6} applied to the Jacobian ring $R^F$, we conclude
$$
T_tV_{[\delta]}=\{P\in \C[x_0,\ldots,x_{n+1}]_d: P\cdot P_\delta\in J^F\}=(J^F:P_\delta)_d.
$$
\end{proof}

In order to prove Theorem \ref{11.10.2018} we will use Corollary \ref{16.10.2018} for $t=0\in T$ corresponding to the Fermat variety, and $\delta=\P^\frac{n}{2}\in \CH^n(X_0)$ a linear cycle inside it.  

\begin{cor}
\label{2.9.2018.2}
Let
$$
X=\{x_0^d+\cdots+x_{n+1}^d=0\}
$$ 
be the Fermat variety. For $\alpha_0,\alpha_2,\ldots,\alpha_n\in \{1,3,\ldots,2d-1\}$ consider 
\begin{equation}
\label{17.10.2018.2}
\P_\alpha^{\frac{n}{2}}:=\{x_0-\zeta_{2d}^{\alpha_0}x_1=\cdots=x_n-\zeta_{2d}^{\alpha_n}x_{n+1}=0\},    
\end{equation}
and $\delta:=\P_\alpha^\frac{n}{2}$. Its associated polynomial is 
\begin{equation}
\label{cycllinfer}
P_\delta=d^{\frac{n}{2}+1}\zeta_{2d}^{\alpha_0+\cdots+\alpha_n}\prod_{j=1}^{\frac{n}{2}+1}\left(\sum_{l=0}^{d-2}x_{2j-2}^{d-2-l}\zeta_{2d}^{\alpha_{2j-2}l}x_{2j-1}^l\right).   
\end{equation}
In particular 
$$
\P^\frac{n}{2}_\alpha\cdot\P^\frac{n}{2}_\beta=\frac{1-(1-d)^{m+1}}{d}
$$
where $m=\text{dim }\P^\frac{n}{2}_\alpha\cap\P^\frac{n}{2}_\beta$.
\end{cor}

\begin{proof}
Computing the Jacobian matrix of $H$ as in Theorem \ref{15.5.2018.2}, we see it is diagonal by $2\times2$ blocks, and each block has determinant
$$
\frac{d(\zeta_{2d}^{\alpha_{2j-2}}x_{2j-2}^{d-1}+x_{2j-1}^{d-1})}{x_{2j-2}-\zeta_{2d}^{\alpha_{2j-2}}x_{2j-1}},
$$
and so \eqref{cycllinfer} follows. In order to compute the intersection product apply Corollary \ref{4.8.2018.3}, part (ii). We just need to compute $c\in\C$ such that $P_\delta\cdot P_\mu\equiv c\cdot d^{n+2}(d-1)^{n+2}(x_0\cdots x_{n+1})^{d-2}$ (mod $\langle x_0^{d-1},\ldots,x_{n+1}^{d-1}\rangle$), where $\delta=\P^\frac{n}{2}_\alpha$ and $\mu=\P^\frac{n}{2}_\beta$. It follows from \eqref{cycllinfer} that 
$$
c=\frac{\zeta_{2d}^{(\alpha_0+\beta_0)+\cdots+(\alpha_n+\beta_n)}}{(d-1)^{n+2}}\prod_{j=1}^{\frac{n}{2}+1}\left(\sum_{l=0}^{d-2}\zeta_{2d}^{\alpha_{2j-2}l+\beta_{2j-2}(d-2-l)}\right)=\frac{\prod_{j=1}^{\frac{n}{2}+1}\left(\sum_{l=0}^{d-2}\zeta_{2d}^{\alpha_{2j-2}(l+1)+\beta_{2j-2}(d-1-l)}\right)}{(d-1)^{n+2}}.
$$
For every $j=1,\ldots,\frac{n}{2}+1$ 
$$
\sum_{l=0}^{d-2}\zeta_{2d}^{\alpha_{2j-2}(l+1)+\beta_{2j-2}(d-1-l)}=-\sum_{l=1}^{d-1}\zeta_{2d}^{(\alpha_{2j-2}-\beta_{2j-2})l}=\left\{
	\begin{array}{ll}
		1-d  & \mbox{if } \alpha_{2j-2}=\beta_{2j-2}, \\
		1 & \mbox{if } \alpha_{2j-2}\neq\beta_{2j-2}.
	\end{array}
\right.
$$
Therefore $c(d-1)^{n+2}=(1-d)^{m+1}$ and so by \eqref{intform} the result follows.
\end{proof}

We end this section by computing the periods of linear cycles inside Fermat varieties. This was the main theorem in \cite[Theorem 1]{MV}. Consider the following set $$I_{(d-2)(\frac{n}{2}+1)}:=\{(i_0,\ldots,i_{n+1})\in \{0,\ldots,d-2\}^{n+2}: i_0+\cdots+i_{n+1}=(d-2)(\frac{n}{2}+1)\},$$ we define for every $i\in I_{(d-2)(\frac{n}{2}+1)}$
$$
\omega_i:=\res\left(\frac{x^i\Omega}{F^{\frac{n}{2}+1}}\right)=\frac{1}{\frac{n}{2}!}\left\{\frac{x^i\Omega_J}{F_J}\right\}_{|J|=\frac{n}{2}}\in H^\frac{n}{2}(X,\Omega_X^\frac{n}{2}).
$$
From Griffiths' work \cite{gr69} we know these forms are a basis for $H^{\frac{n}{2},\frac{n}{2}}(X)_{\prim}$. 

\begin{cor}[\cite{MV}]
\label{perlinfer}
Let $X\subseteq \P^{n+1}$ be the degree $d$ even dimensional Fermat variety, let $\P^\frac{n}{2}\subseteq X$ as in \eqref{17.10.2018.2} for $\alpha_0=\cdots=\alpha_n=1$, and let $i\in I_{(d-2)(\frac{n}{2}+1)}$. Then
$$
\int_{\P^{\frac{n}{2}}}\omega_i= \left\{
	\begin{array}{ll}
		   \frac{(2\pi \sqrt{-1} )^
		   {\frac{n}{2}}}{d^{\frac{n}{2}+1}\cdot \frac{n}{2}!} {\zeta_{2d}}^{\frac{n}{2}+1+i_0+i_2+\cdots+i_n}  & \mbox{if } i_{2l-2}+i_{2l-1}=d-2, \forall l=1,\ldots,\frac{n}{2}+1, \\
		0 & \mbox{otherwise. } 
	\end{array}
\right.
$$
\end{cor}

\begin{proof}
By Theorem \ref{15.5.2018.2} we just need to compute $c\in \C$ such that 
$$
x^iP_\delta\equiv c\cdot d^{n+2}(d-1)^{n+2}(x_0\cdots x_{n+1})^{d-2}\text{ (mod }\langle x_0^{d-1},\ldots,x_{n+1}^{d-1}\rangle).
$$
By Proposition \ref{2.9.2018.2} 
$$
x^iP_\delta=d^{\frac{n}{2}+1}\zeta_{2d}^{\frac{n}{2}+1}x^i\prod_{j=1}^{\frac{n}{2}+1}\left(\sum_{l=0}^{d-2}x_{2j-2}^{d-2-l}\zeta_{2d}^lx_{2j-1}^l\right),
$$
$$
\equiv c_i\cdot (x_0\cdots x_{n+1})^{d-2} \text{ (mod }\langle x_0^{d-1},\ldots,x_{n+1}^{d-1}\rangle ).
$$
If for every $j=1,\ldots,\frac{n}{2}+1$ there exist $l_j\in \{0,\ldots,d-2\}$ such that $l_j+i_{2j-1}=d-2$ and $d-2-l_j+i_{2j-2}=d-2$, then $c_i=\zeta_{2d}^{\frac{n}{2}+1+l_1+\cdots+l_{\frac{n}{2}}}$. This condition is equivalent to $l_j=i_{2j-2}$ and $i_{2j-2}+i_{2j-1}=d-2$. Otherwise $c_i=0$, and the result follows.
\end{proof}

\section{Proof of Theorem \ref{11.10.2018}}
\label{4}
Let $T$ be the parameter space of smooth degree $d$ hypersurfaces of $\P^{n+1}$, of even dimension $n$. Let $0\in T$ be the point corresponding to the Fermat variety $X_0=\{x_0^d+\cdots+x_{n+1}^d=0\}$. Letting 
$$
\P^{n-m}:=\{x_{n-2m}-\zeta_{2d}x_{n-2m+1}=\cdots=x_n-\zeta_{2d}x_{n+1}=0\},
$$
$$
\P^\frac{n}{2}:=\{x_0-\zeta_{2d}x_1=\cdots=x_{n-2m-2}-\zeta_{2d}x_{n-2m-1}=0\}\cap \P^{n-m},
$$
$$
\check\P^\frac{n}{2}:=\{x_0-\zeta_{2d}^{\alpha_0}x_1=\cdots=x_{n-2m-2}-\zeta_{2d}^{\alpha_{n-2m-2}}x_{n-2m-1}=0\}\cap \P^{n-m},
$$ 
where $\alpha_0,\alpha_2,\ldots,\alpha_{n-2m-2}\in \{3,\ldots,2d-1\}$. Then 
$$
\P^m:=\P^\frac{n}{2}\cap\check\P^\frac{n}{2}=\{x_0=x_1=\cdots=x_{n-2m-1}=0\}\cap\P^{n-m}.
$$
The following result is due to Movasati \cite[Propositions 17.9 and 17.9]{ho13}.

\begin{prop}
\label{16.10.2018.2}
$
\text{dim }V_{[\P^\frac{n}{2}]}\cap V_{[\check\P^\frac{n}{2}]}=\text{dim }T_0V_{[\P^\frac{n}{2}]}\cap T_0V_{[\check\P^\frac{n}{2}]}.
$
\end{prop}

Movasati's proof of Proposition \ref{16.10.2018.2} consists in computing explicitly both sides of the equality. Since variational Hodge conjecture holds for linear cycles (see \cite{Dan2017}, \cite{GMCD-NL}, or \cite{MV}), $V_{[\P^\frac{n}{2}]}\cap V_{[\check\P^\frac{n}{2}]}$ corresponds to the locus of hypersurfaces containing two linear cycles intersecting each other in a $m$ dimensional linear subvariety. Knowing this, it is easy to compute its dimension as a fibration over the incidence variety of pairs of $\frac{n}{2}$-dimensional linear subvarieties of $\P^{n+1}$ intersecting each other in a $m$-dimensional linear subvariety. In fact, (this computation can be found in \cite[Proposition 17.9]{ho13})  
\begin{equation}
\label{11.10.2018.2}
\text{Codim }V_{[\P^\frac{n}{2}]}\cap V_{[\check\P^\frac{n}{2}]}=2{\frac{n}{2}+d\choose d}-2(\frac{n}{2}+1)^2-{m+d\choose d}+(m+1)^2.
\end{equation}
On the other hand, it is also easy to compute the codimension of $T_0V_{[\P^\frac{n}{2}]}\cap T_0V_{[\check\P^\frac{n}{2}]}$ (see \cite[Proposition 17.8]{ho13}) and coincides with \eqref{11.10.2018.2}
\begin{eqnarray*}
\text{Codim }T_0V_{[\P^\frac{n}{2}]}\cap T_0V_{[\check\P^\frac{n}{2}]}&=&2\text{Codim }T_0V_{[\P^\frac{n}{2}]}-\text{Codim } T_0V_{[\P^\frac{n}{2}]}+ T_0V_{[\check\P^\frac{n}{2}]}\\
&=&2{\frac{n}{2}+d\choose d}-2(\frac{n}{2}+1)^2-{m+d\choose d}+(m+1)^2.
\end{eqnarray*}

\begin{rmk}
\label{16.10.2018.4}
After Proposition \ref{16.10.2018.2}, Theorem \ref{11.10.2018} is reduced to show that
$$
T_0V_{[\P^\frac{n}{2}]}\cap T_0V_{[\check\P^\frac{n}{2}]}=T_0V_{[\delta]},
$$
if and only if $m<\frac{n}{2}-\frac{d}{d-2}$.
By Corollaries \ref{16.10.2018} and \ref{2.9.2018.2}, this is equivalent to the following algebraic equality 
\begin{equation}
\label{16.10.2018.3}
(J^F:P_1)_d\cap (J^F:P_2)_d=(J^F:P_1+P_2)_d.
\end{equation}
Where $P_1=R_1Q$, $P_2=R_2Q$,
$$
Q:=\prod_{k\ge n-2m\text{ even}}\frac{(x_k^{d-1}-(\zeta_{2d}x_{k+1})^{d-1})}{(x_k-\zeta_{2d}x_{k+1})},
$$
$$
R_1:=c_1\cdot\prod_{k<n-2m\text{ even}}\frac{(x_k^{d-1}-(\zeta_{2d}x_{k+1})^{d-1})}{(x_k-\zeta_{2d}x_{k+1})},
$$
and 
$$
R_2:=c_2\cdot\prod_{k<n-2m\text{ even}}\frac{(x_k^{d-1}-(\zeta_{2d}^{\alpha_k}x_{k+1})^{d-1})}{(x_k-\zeta_{2d}^{\alpha_k}x_{k+1})},
$$
for some $c_1,c_2\in\C^\times$.
\end{rmk}

\noindent{\textbf{Proof of Theorem \ref{11.10.2018}}}
After Remark \ref{16.10.2018.4} we have reduced the proof to prove the equality \eqref{16.10.2018.3}. We claim that
\begin{equation}
\label{3.12.2018.2}
(J^F:P_1)_e\cap (J^F:P_2)_e=(J^F:P_1+P_2)_e,
\end{equation}
if and only if $e<(d-2)(\frac{n}{2}-m)$ or $e>(d-2)(\frac{n}{2}+1)$. In fact, for $e>(d-2)(\frac{n}{2}+1)$ the claim follows from the fact that $(d-2)(\frac{n}{2}+1)$ is the socle of the three ideals appearing in \eqref{3.12.2018.2}. For $e<(d-2)(\frac{n}{2}-m)$, consider any $q\in (J^F:P_1+P_2)_e$. Write
$$
q=r+s,
$$
where $r\in \C[x_0,\ldots,x_{n-2m-1}]_e$ and $s\in \langle x_{n-2m},\ldots,x_{n+1}\rangle_e\subseteq\C[x_0,\ldots,x_{n+1}]_e$. Noting that $$(J^F:Q)=\langle x_0^{d-1},\ldots,x_{n-2m-1}^{d-1},x_{n-2m},\ldots,x_{n+1}\rangle,$$ it is clear that $s\in (J^F:P_i)=((J^F:Q):R_i)$ for every $i=1,2$. In consequence $r\in ((J^F:Q):R_1+R_2)$. Since $r\cdot(R_1+R_2)$ does not depend on $x_{n-2m},\ldots,x_{n+1}$ we conclude that $$r\in (I:R_1+R_2)_e\subseteq\C[x_0,\ldots,x_{n-2m-1}]_e$$ for $I=\langle x_0^{d-1},\ldots,x_{n-2m-1}^{d-1}\rangle$. Using Proposition \ref{caex} for $r=\frac{n}{2}-m$, we conclude that $r\in (I:R_i)_e$ for $i=1,2$, and so $q\in (J^F:P_i)_e$ for $i=1,2$ as claimed.

Finally, if $(d-2)(\frac{n}{2}-m)\le e\le (d-2)(\frac{n}{2}+1)$, we know from Proposition \ref{caex} for $r=\frac{n}{2}-m$, that there exist some $p\in \C[x_0,\ldots,x_{n-2m}]$ such that
$$
p\in (J^F:R_1+R_2)_{(d-2)(\frac{n}{2}-m)}\setminus (J^F:R_1)_{(d-2)(\frac{n}{2}-m)}, 
$$
and so
$$
p\in (J^F:P_1+P_2)_{(d-2)(\frac{n}{2}-m)}\setminus (J^F:P_1)_{(d-2)(\frac{n}{2}-m)}. 
$$
Since $(J^F:P_1)$ is Artinian Gorenstein with socle $(d-2)(\frac{n}{2}+1)$, we conclude that there exist some $q\in \C[x_0,\ldots,x_{n+1}]_{e-(d-2)(\frac{n}{2}-m)}$ such that 
$$
pq\in (J^F:P_1+P_2)_e\setminus (J^F:P_1)_e, 
$$
as desired.
\qed

\section{Proof of Theorem \ref{gent}}

Let $\pi:X\rightarrow T$ be the family of smooth degree $d$ hypersurfaces of $\P^{n+1}$, of even dimension $n$. Consider
$$
W_m:=\{t\in T: X_t\text{ contains two linear cycles }\P^\frac{n}{2}, \check\P^\frac{n}{2}\text{ with }\P^\frac{n}{2}\cap\check\P^\frac{n}{2}=\P^m\}.    
$$
Let $\mathscr{H}$ be the incidence variety (between elements of the Hilbert scheme) parametrizing triples $(X_t,\P^\frac{n}{2},\check\P^\frac{n}{2})$, where $X_t$ is a smooth degree $d$ hypersurface of $\P^{n+1}$ containing two linear subvarieties $\P^\frac{n}{2},\check\P^\frac{n}{2}\subseteq X_t$ such that $\P^\frac{n}{2}\cap \check\P^\frac{n}{2}=\P^m$. Consider the map
$$
\Phi:\mathscr{H}\rightarrow Hom_\C(\C[x_0,\ldots,x_{n+1}]_d,\C[x_0,\ldots,x_{n+1}]_{d\frac{n}{2}-n-2}^*),
$$
given by
$$
\Phi(X_t,\P^\frac{n}{2},\check\P^\frac{n}{2})= \left[\int_{a\cdot\P^\frac{n}{2}+b\cdot\check\P^\frac{n}{2}}\res\left(\frac{PQ\Omega}{F^{\frac{n}{2}+1}}\right)\right]_{P,Q}.
$$
This map is regular (by \cite{cadeka} or by Theorem \ref{15.5.2018.2}), hence it is continuous in the Zariski topology of $\mathscr{H}$. By Proposition \ref{17.10.2018} we already know that
$$
T_tV_{a[\P^\frac{n}{2}]+b[\check\P^\frac{n}{2}]}=\text{Ker }\Phi(X_t,\P^\frac{n}{2},\check\P^\frac{n}{2}).
$$
This implies that each subset of $\mathscr{H}$ where $\text{Ker }\Phi$ has constant dimension is a locally closed subset in Zariski topology. Theorem \ref{11.10.2018} implies that $$\text{dim }\mathscr{H}=\text{dim }T_0V_{a[\P^\frac{n}{2}]+b[\check\P^\frac{n}{2}]}.$$ Therefore, the point $(X_0,\P^\frac{n}{2},\check\P^\frac{n}{2})$ corresponding to the Fermat variety together with its two linear subvarieties, is a smooth point of $\mathscr{H}$. Furthermore $\text{Ker }\Phi$ has constant dimension in a polydisc around $(X_0,\P^\frac{n}{2},\check\P^\frac{n}{2})\in\mathscr{H}$, hence the same holds in a Zariski neighbourhood of this point. Therefore $$\text{dim }\mathscr{H}=\text{dim }T_tV_{a[\P^\frac{n}{2}]+b[\check\P^\frac{n}{2}]} \hspace{8mm} \forall t\in U,$$ for $U$ a Zariski open set of $W_m$ (not necessarily a neighbourhood of the Fermat variety), and variational Hodge conjecture holds for $[\delta]=a[\P^\frac{n}{2}]+b[\check\P^\frac{n}{2}]\in H^n(X_t,\Z)\cap H^{\frac{n}{2},\frac{n}{2}}(X_t)$ and $t\in U$.

\section{Final remarks}
\label{fin}

Considering $[\delta]=a[\P^\frac{n}{2}]+b[\check\P^\frac{n}{2}]\in H^n(X_0,\Z)\cap H^{\frac{n}{2},\frac{n}{2}}(X_0)$, we can ask if this Hodge cycle satisfies variational Hodge conjecture for the cases not covered by Theorem \ref{11.10.2018}. The remaining open cases are the following: $(d,m)=(3,\frac{n}{2}-3), (3,\frac{n}{2}-2), (4,\frac{n}{2}-2)$ and $m=\frac{n}{2}-1$ with $a\neq b$. Note that the cases $m=\frac{n}{2}-1$ with $a=b$, and $m=\frac{n}{2}$ are both complete intersection algebraic cycles, where variational Hodge conjecture holds by \cite{Dan2017}. It would be interesting to determine whether for the remaining cases the corresponding Hodge locus $V_{[\delta]}$ is smooth and reduced. This problem has been considered by Movasati for small degree and dimension in \cite[Chapter 18]{ho13}. After knowing all the periods of the linear cycles inside Fermat, Movasati was able to compute higher order approximations of the Hodge locus. Using them, he proves in several cases (see \cite[Theorem 18.3]{ho13}) that $V_{[\delta]}$ is not smooth and reduced (explaining the difference between the tangent spaces of $V_{[\delta]}$ and $V_{[\P^\frac{n}{2}]}\cap V_{[\check\P^\frac{n}{2}]}$ in Theorem \ref{11.10.2018}). On the other hand he also provides interesting examples, such as $(n,d,m,a,b)=(6,3,1,1,-1)$ and $(n,d,m)=(6,3,0)$ (see \cite[Theorem 18.2]{ho13}), where $V_{[\delta]}$ is possibly smooth and reduced. In such cases $V_{[\delta]}$ must be strictly bigger than $V_{[\P^\frac{n}{2}]}\cap V_{[\check\P^\frac{n}{2}]}$, and it would be very interesting to study this phenomenon, and determine if it is due to the existence of some new algebraic cycle (with cohomological class having the same primitive part as $[\delta]$) with a larger deformation space. 

\section*{Acknowledgements}

Part of this work was developed during my Ph.D. at IMPA between 2016 and 2018. Part of this article was written during my short stay at Hausdorff Research Institute for Mathematics during the program ``Periods in Number Theory, Algebraic Geometry and Physics", and also during my visit to Harvard CMSA. I thank all these institutes for their support and for providing such stimulating environments to work. I am deeply grateful to my advisor Hossein Movasati, for all his comments and contributions to this article. I am also grateful to Emre Sert\"oz, Ananyo Dan, Prof. Remke Kloosterman and Prof. Claire Voisin for their suggestions and criticism that inspired several improvements on the final presentation of this work.


\end{document}